\documentclass[12pt,twoside]{article}
\usepackage[all]{}
\usepackage {amssymb,latexsym,amsthm,amscd,amsmath}
\newtheorem{theorem}{Theorem}[section]
\newtheorem{lemma}{Lemma}[section]
\newtheorem{corollary}{Corollary}[section]
\newtheorem{proposition}{Proposition}[section]

\newtheorem*{theorem*}{Theorem}
\theoremstyle{definition}

\numberwithin{equation}{subsection}

\newcommand{\ignore}[1]{}

\newcommand{\mynote}[1]{}
\begin{document}
\setcounter{section}{0}
\title{\bf $R$-triviality of some exceptional groups}
\author{Maneesh Thakur \\ \small {Indian Statistical Institute, 7-S.J.S. Sansanwal Marg} \\ \small {New Delhi 110016, India} \\ \small {e-mail: maneesh.thakur@
gmail.com}}
\date{}
\maketitle
\begin{abstract}
\noindent
\it{The main aim of this paper is to prove $R$-triviality for simple, simply connected algebraic groups with Tits index $E_{8,2}^{78}$ or $E_{7,1}^{78}$, defined over a field $k$ of arbitrary characteristic. Let $G$ be such a group. We prove that there exists a quadratic extension $K$ of $k$ such that $G$ is $R$-trivial over $K$, i.e., for any extension $F$ of $K$, $G(F)/R=\{1\}$, where $G(F)/R$ denotes the group of $R$-equivalence classes in $G(F)$, in the sense of Manin (see \cite{M}). As a consequence, it follows that the variety $G$ is retract $K$-rational and that the Kneser-Tits conjecture holds for these groups over $K$. Moreover,  
$G(L)$ is projectively simple as an abstract group for any field extension $L$ of $K$. In their monograph (\cite{TW}) J. Tits and Richard Weiss conjectured that for an Albert division algebra $A$ over a field $k$, its structure group $Str(A)$ is generated by scalar homotheties and its $U$-operators. This is known to be equivalent to the Kneser-Tits conjecture for groups with Tits index $E_{8,2}^{78}$. We settle this conjucture for Albert division algebras which are first constructions, in affirmative. 
These results are obtained as corollaries to the main result, which shows that if $A$ is an Albert division algebra which is a first construction and $\Gamma$ its structure group, i.e., the algebraic group of the norm similarities of $A$, 
then $\Gamma(F)/R=\{1\}$ for any field extension $F$ of $k$, i.e., $\Gamma$ is $R$-trivial.}
\end{abstract} 

\section{\bf Introduction}
The main aim of this paper is to prove $R$-triviality of certain exceptional algebraic groups. More precisely, let $G$ be a simple, simply connected algebraic group defined over a field $k$ of arbitrary characteristic, with Tits index $E^{78}_{8,2}$ or $E^{78}_{7,1}$ (see \cite{T3} for the index notation). Groups with these indices exist over a field $k$ if $k$ admits a central division algebra of degree $3$ with non-surjective reduced norm, or if $k$ admits degree $3$ central division algebras with unitary involutions over a quadratic extension of $k$, with non-surjective reduced norm map (see \cite{T1}, 3.2, 3.3.1). Note also that the groups with Tits indices $E^{78}_{8,2},~E^{78}_{7,1}$ do not exist over finite fields, number fields, algebraically closed fields, local fields, field of reals. However, these exist, for example, over $\mathbb{Q}(t)$, the function field over $\mathbb{Q}$ in one variable.

Groups with index $E^{78}_{8,2}$ are classified by isotopy classes of Albert division algebras over $k$ and their associated buildings are Moufang hexagons defined by a hexagonal system of type $27/K$ where $K$ is either a quadratic extension of $k$ or $K=k$ (see \cite{TW}, \cite{TH}). 
Groups with index $E^{78}_{7,1}$ are classified by isotopy classes of Albert division algebras and they are associated to Moufang sets $\mathbb{M}(A)$, for $A$ an Albert division algebra (e.g. see \cite{TS} or \cite{BDS}). 

Let $G$ be an algebraic group over $k$, with Tits index as one of the above mentioned indices. We prove that there exists a quadratic field extension $K$ of $k$ such that $G$ is $R$-trivial over $K$, i.e., for any field extension $L$ of $K$, $G(L)/R=\{1\}$ in the sense of Manin (\cite{M}). Since $E_8$ has trivial center, it follows that when $G$ has index $E^{78}_{8,2}$, there is a quadratic extension $K$ of $k$ such that $G(L)$ is a simple (abstract) group for any field extension $L$ of $K$. It is well known (see \cite{T1}) that the anisotropic kernels of these groups correspond (upto a central torus) to the structure groups of Albert division algebras over $k$. 

Albert division algebras over a field $k$ are obtainable from two (rational) constructions due to Tits, called the {\bf first construction} and the {\bf second construction}. We prove that if the anisotropic kernel of $G$ (for $G$ as above) comes from a first construction Albert division algebra, then $G$ is $R$-trivial over $k$. In this situation $G$ splits over a cubic extension of $k$. 

The {\bf Kneser-Tits conjecture} predicts that if $G$ is a connected, simple, simply 
connected algebraic group, defined and isotropic over a field $k$, then $G(k)/G(k)^{\dagger}=\{1\}$, where $G(k)^{\dagger}$ is the subgroup of $G(k)$ generated by the $k$-rational points of the unipotent radicals of $k$-parabolics of $G$ (see \cite{PRap}, \S 7.2, \cite{G}). By a theorem of Gille, this is equivalent to proving $G(k)/R=\{1\}$ (\cite{G}, Thm. 7.2). 

A conjecture of J. Tits and Richard M. Weiss asserts that if $A$ is an Albert division algebra over a field $k$, then its structure group $Str(A)$ is generated by scalar homotheties and $U$-operators (defined in \S 2) of $A$ (see \cite{TW}, 37.41, 37.42 and page 418). This conjecture is equivalent to the Kneser-Tits conjecture for groups with Tits index $E^{78}_{8,2}$ (see \cite{TW}). 

Our results prove the {\bf Tits-Weiss conjecture} when $A$ is a first Tits construction. For an extensive survey and new results on the Kneser-Tits problem, we refer the reader to (\cite{G}) along with the classics (\cite{T0}, {\cite{T2}) and (\cite{PRap}) and the paper (\cite{P-R1}) for reduction to the relative rank one case. We mention some recent results: 
The case of simply connected group with index $^2E^{29}_{6,1}$ was settled in (\cite{Gar}), where a modern version (in terms of $W$-triviality) of F. D. Veldkamp's proof (\cite{V1}, \cite{V2}) is presented. 
The case of the trialitarian forms of $D_4$ over global fields was settled in (\cite{Pra}). The paper {\cite{Pra}) also presents a description of the Whitehead group (over an infinite field $k$) of an absolutely simple, simply connected algebraic group of type $^{3,6}D_4$ of $k$-rank $1$.  
The result for such groups over a general field is covered in (\cite{G}, \S 6.1). The Kneser-Tits conjecture for groups with Tits index $E^{66}_{8,2}$ over an arbitrary field was settled in (\cite{P-T-W}). 

This paper is the second in a series of papers dedicated to this problem. In the first paper (\cite{Th-1}) we proved, among other results, that the conjecture is true for {\bf pure first constructions}. We had also proved that for such an Albert division algebra $A$ over $k$, the group {\bf Aut}$(A)$ is $R$-trivial. 

The problem of $R$-equivalence for semisimple algebraic groups has been studied by several mathematicians, see for example (\cite {CM}, \cite{CM-2}, \cite{CS}, \cite{CT}, \cite{B-S}) and the references therein. This connects with some algebro-geometric properties of algebraic groups; e.g. if $G$ is a simple, simply connected algebraic group, isotropic and defined over a field $k$, then $G$ is $R$-trivial if and only if the variety of $G$ is retract $k$-rational (see \cite{G}, Thm. 5.9). 
\vskip1mm
\noindent
{\bf Remark :} For a simple, simply connected algebraic group $G$ with Tits index $E^{78}_{8,2}$ or $E^{78}_{7,1}$, the (semisimple) anisotropic kernel $H$ of $G$ is isomorphic to {\bf Isom}$(A)$, the full group of isometries of the norm form of an Albert division algebra $A$ over $k$. The anisotropic kernel $H$ of $G$ is split by a cubic extension of $k$ if and only if $A$ is a first Tits construction (see Proposition 2.2 and Corollary 2.1). 
\vskip1mm 
\noindent
{\bf Retract rationality :} By (\cite{G}), our results also prove that if $G$ is a simple, simply connected algebraic group defined over $k$ with Tits index $E_{8,2}^{78}$ or $E_{7,1}^{78}$, then 
there exists a quadratic extension $K$ of $k$ such that the underlying variety of $G$ is retract $K$-rational. If the anisotropic kernel of $G$ is split over a cubic extension of $k$, then $G$ is retract $k$-rational.     
\vskip1mm
\noindent
{\bf Plan of the paper :} In this paper, $k$ will always denote a base field, infinite of arbitrary characteristic. 

In \S 2, we give a quick introduction to Albert algebras, Tits processes and the structure group of an Albert algebra. A description of strongly inner groups of type $E_6$ and those with Tits index $E^{78}_{8,2}$ and $E^{78}_{7,1}$ is given in detail. 
This is followed by a brief introduction to $R$-equivalence on algebraic groups. At the end of \S 2, we state our main results and some of their consequences. 

\S 3 and \S 4 contain some of the key results of the paper. \S 3 is mainly concerned with extension of automorphisms of a $9$-dimensional subalgebra of an Albert algebra $A$ to an automorphism of $A$ and also discusses $R$-triviality of some subgroups of 
{\bf Aut}$(A)$. \S 4 contains results on extending norm similarities of $9$-dimensional subalgebras to those of the Albert algebra (see Theorem \ref{ext-norm-sim}). Some of the most important results on $R$-triviality are proved \S 5. Here we develop the $R$-triviality results on various subgroups of {\bf Aut}$(A)$ and {\bf Str}$(A)$ for Albert algebras in general as well as for the special case of first constructions. These results culminate in the main theorem (Theorem \ref{Main}) , where we prove {\bf Str}$(A)$ is $R$-trivial for first Tits construction Albert algebras. \S 6 discusses a proof of the $R$-triviality of groups of type $E^{78}_{8,2}$, with anisotropic kernel split over a cubic extension. A proof of the Tits-Weiss conjecture for Albert division algebras, arising from first Tits construction is presented (Corollary \ref{TW}). In \S 7, we discuss groups with Tits index $E^{78}_{7,1}$, having anisotropic kernel split over a cubic extension. We prove that such groups are $R$-trivial. We end \S with some $R$-triviality results for reduced Albert algebras. 

In \S 8 some retract rationality results are derived. We prove, as a consequence of our results in \S 6 and \S 7, that any group $G$ whose Tits index is $E^{78}_{7,1}$ or $E^{78}_{8,2}$ and the anisotropic kernel is split over a cubic extension of $k$, is retract $k$-rational.  

\section{\bf Preliminaries and results} For much of the preliminary material on Albert algebras, we refer the reader to (\cite{P2}), (\cite{SV}) and (\cite{KMRT}). We recall below, very briefly, some notions we will need. All base fields considered in this paper will be 
assumed to be infinite and of arbitrary characteristic. To define the notion of an Albert algebra over a field of arbitrary characteristics, we take the approach via cubic norm structure, as defined in (\cite{PR7} or \cite{PR5}) .\\
\noindent
{\bf Cubic norm structures :} Let $J$ be a finite dimensional vector space over a field $k$. A cubic norm structure on $J$ consists of a tuple $(N,\#, c)$, where $c\in J$ is a base point, called the {\bf identity element} of the norm structure and 
\begin{enumerate} 
\item $N:J\rightarrow k$ is a {\bf cubic form} on $J$,
\item $N(c)=1$,
\item the {\bf trace form} $T:J\times J\rightarrow k$, defined by $T(x,y):=-\Delta^x_c\Delta^y log N$, is nondegenerate,
\item the {\bf adjoint} map $\#:J\rightarrow J$ defined by $T(x^{\#},y)=\Delta^y_x N$, is a {\bf quadratic map} such that 
\item $x^{\#\#}=N(x)x$,
\item $c^{\#}=c$,
\item $c\times x=T(x)c-x$, where $T(x):=T(x,c)$ is the {\bf trace} of $x$ and $x\times y:=(x+y)^{\#}-x^{\#}-y^{\#}$,
\end{enumerate}
and these hold in all scalar extensions of $J$. In the above, $\Delta^y_xf$ denotes the directional derivative of a polynomial function $f$ on $J$, in the direction $y$, evaluated at $x$. For differential calculus of rational maps and related notation, we refer to (\cite{J1}, Chap. VI, Sect. 2). Let, for $x\in J$, 
$$U_x(y):=T(x,y)x-x^{\#}\times y,~y\in J.$$ 
Then with $1_J:=c$ and $U_x$ as above, one gets a unital {\bf quadratic Jordan algebra} structure on $J$ (see \cite{McK}), sometimes denoted by $J(N,c)$. The linear operators $U_x$ as defined above, are called the {\bf $U$-operators} of the Jordan algebra $J$. An element $x\in J$ is {\bf invertible} if and only if $N(x)\neq 0$ and in that case one has $x^{-1}=N(x)^{-1}x^{\#}$. Recall that $J(N,c)$ is called a {\bf division algebra} if $U_x$ is surjective for all $x\neq 0$, which is equivalent to $N(x)\neq 0$ for all $x\neq 0$. We list some examples and notation that we will use in the paper: \\
\vskip1mm
\noindent
{\bf Example 1.} Let $D$ denote a separable associative algebra over $k$, of degree $3$. Let $N_D$ denote its norm and $T_D$ the trace. Then with $1_D$ the unit element of $D$ and $\#:D\rightarrow D$ the (classical) adjoint map, we get a quadratic Jordan algebra structure on $D$, which we will denote by $D_+$. \\
\noindent
{\bf Example 2.} Let $(B,\sigma)$ be a separable associative algebra over $k$ with an involution $\sigma$ of the second kind (over its center). With the unit element $1$ of $B$ and the restriction of the norm $N_B$ of $B$ to $(B,\sigma)_+:=\{b\in B|\sigma(b)=b\}$, we get a cubic norm structure on $(B,\sigma)_+$ which is a substructure of $B_+$. \\
\noindent
{\bf Example 3.} Let $C$ be an octonion algebra over a field $k$ with $n$ as its norm map and $t$ its trace. Let $x\mapsto \overline{x}$ denote its canonical involution. 
Let $\gamma=\text{diag}(\gamma_1,\gamma_2,\gamma_3)\in GL_3(k)$ and 
$$\mathcal{H}_3(C,\gamma):=\{x=(x_{ij})\in M_3(C)|\gamma^{-1}\overline{x}^t\gamma=x~\text{and}~x_{ii}\in k,~1\leq i\leq 3\},$$ 
where $M_3(C)$ is the algebra of $3\times 3$ matrices with entries in $C$ and $x^t$ is the transpose of $x$. It follows that any $x\in\mathcal{H}_3(C,\gamma)$ has the form 
$$x=\left(\begin{matrix}
\alpha_1&\gamma_2c&\gamma_3\overline{b}\\
\gamma_1\overline{c}&\alpha_2&\gamma_3a\\ 
\gamma_1b&\gamma_2\overline{a}&\alpha_3\end{matrix}\right) ,$$
where $\alpha_i\in k,~ 
a,b,c\in C$. Define 
$$N(x):=\alpha_1\alpha_2\alpha_3-\gamma_2\gamma_3\alpha_1n(a)-\gamma_3\gamma_1\alpha_2n(b)-\gamma_1\gamma_2\alpha_3n(c)+\gamma_1\gamma_2\gamma_3 t(abc),$$
$$1:=e_{11}+e_{22}+e_{33},$$ 
and 
$$x^{\#}:=\alpha_2\alpha_3-\gamma_2\gamma_3 n(a)+\alpha_3\alpha_1-\gamma_3\gamma_1 n(b)+\alpha_1\alpha_2-\gamma_1\gamma_2 n(c)~~~~~~~~~~~~~~~~$$
$$+(\gamma_1\overline{bc}-\alpha_1a)[23]+(\gamma_2\overline{ca}-\alpha_2b)[31]+(\gamma_3\overline{ab}-\alpha_3c)[12],~~~~~~~~~~~$$
where $a[ij]:=\gamma_ja e_{ij}+\gamma_i\overline{a}e_{ji}$ in terms of the matrix units. This defines a cubic norm structure (see \cite{PR5}), hence a Jordan algebra, called a {\bf reduced Albert algebra} over $k$. We call a reduced Albert algebra {\bf split} if the coordinate octonion algebra is split, in which case 
$$\mathcal{H}_3(C,\gamma)\cong \mathcal{H}_3(C):=\mathcal{H}_3(C,1).$$

\vskip1mm
\noindent
{\bf Tits process :} {\bf 1.} Let $D$ be a finite dimensional associative $k$-algebra of degree $3$ with norm $N_D$ and trace $T_D$. Let $\lambda\in k^*$. We define a cubic norm structure on the $k$-vector space $D\oplus 
D\oplus D$ by 
$$1:=(1,0,0),~N((x,y,z)):=N_D(x)+\lambda N_D(y)+\lambda^{-1}N_D(z)-T_D(xyz),$$
$$(x,y,z)^{\#}:=(x^{\#}-yz,\lambda^{-1}z^{\#}-xy,\lambda y^{\#}-zx).$$  
The (quadratic) Jordan algebra associated to this norm structure is denoted by $J(D,\lambda)$. We regard $D_+$ as a subalgebra of $J(D,\lambda)$ through the first summand. Recall (see \cite{PR7}, 5.2) that $J(D,\lambda)$ is a division algebra if and only if $\lambda$ is not a norm from $D$. This construction is called the {\bf first Tits process} arising from the parameters $D$ and $\lambda$. 

{\bf 2.} Let $K$ be a quadratic \'{e}tale extension of $k$ and $B$ be a separable associative algebra of degree $3$ over $K$ with a $K/k$ involution $\sigma$. Let $u\in (B,\sigma)_+$ be a unit such that $N_B(u)=\mu\overline{\mu}$ for some $\mu\in K^*$, where $a\mapsto \overline{a}$ denotes the nontrivial $k$-automorphism of $K$. We define a cubic norm structure on the $k$-vector space $(B,\sigma)_+\oplus B$ by the formulae:
$$N((b,x)):=N_B(b)+T_K(\mu N_B(x))-T_B(bxu\sigma(x)),$$
$$(b,x)^{\#}:=(b^{\#}-xu\sigma(x), \overline{\mu}\sigma(x)^{\#}u^{-1}-bx),~1:=(1_B,0).$$ 
The (quadratic) Jordan algebra corresponding to this cubic norm structure is denoted by $J(B,\sigma,u,\mu)$. We note that $(B,\sigma)_+$ can be identified with a subalgebra of $J(B,\sigma, u,\mu)$ through the first summand. Recall (see \cite{PR7}, 5.2) that $J(B,\sigma,u,\mu)$ is a division algebra if and only if $\mu$ is not a norm from $B$. This construction is called the {\bf second Tits process} arising from the parameters $B,u,\mu$. 

Note that when $K=k\times k$, then $B=D\times D^{\circ}$ and $\sigma$ is the switch involution, where $D^{\circ}$ is the opposite algebra of $D$. In this case, the second construction $J(B,\sigma,u,\mu)$ can be identified with a first construction $J(D,\lambda)$. 
\vskip1mm
\noindent
{\bf Albert algebras :} In the Tits process {\bf (1)} described above, if we start with a central simple algebra $D$ and $\lambda\in k^*$, we get a first Tits construction Albert algebra $A=J(D,\lambda)$ over $k$. 
Similarly, in the Tits process {\bf (2)} above, if we start with a central simple algebra $(B,\sigma)$ with center a quadratic \'{e}tale algebra $K$ over $k$ and an involution $\sigma$ of the second kind, $u,\mu$ as described above, we get a second Tits construction Albert algebra $A=J(B,\sigma,u,\mu)$ over $k$. 
\noindent
One knows that all Albert algebras can be obtained via Tits constructions. 

An Albert algebra is a division algebra 
if and only if its (cubic) norm $N$ is anisotropic over $k$ (see \cite{KMRT}, \S 39). We also recall here that a first construction Albert algebra $A$ is either a division algebra, or it is split, i.e., is isomorphic to $\mathcal{H}_3(C)$ with $C$ the split octonion algebra $k$. Moreover, any Albert division algebra arising from a first construction is split over any of its cubic subfields.

If $A=J(B,\sigma,u,\mu)$ as above, then $A\otimes_kK\cong J(B,\mu)$ as $K$-algebras, where $K$ is the center of $B$ (see \cite{KMRT}, 39.4). 
\vskip1mm
\noindent
{\bf Hexagonal systems :} A hexagonal system is a cubic norm structure whose norm is anisotropic. We say a hexagonal system is of {\bf type} $27/k$ if it arises from a first Tits construction Albert division algebra over $k$, and of type $27/K$, for $K$ a quadratic field extension of $k$, if it arises from a Tits second construction Albert division algebra, corresponding to $(B,\sigma)_+$ and $K$ is the center of $B$.
\vskip1mm
\noindent  
{\bf Pure first and second constructions :} Let $A$ be an Albert algebra over $k$. If $A$ is a first construction, but not a second construction with respect to a degree $3$ central simple algebra whose center is a quadratic field extension of $k$, then we call $A$ a pure first construction Albert algebra. Similarly pure second construction Albert algebras are those which are not obtainable through the first construction. 
\vskip1mm
\noindent
We recall at this stage that if $A$ is an Albert division algebra, then any subalgebra of $A$ is either $k$ or a cubic subfield of $A$ or of the form $(B,\sigma)_+$ for a degree $3$ central simple algebra $B$ with an involution $\sigma$ of the second kind over its center $K$, a quadratic \'{e}tale extension of $k$ (see \cite{J1}, Chap. IX, \S 12, Lemma 2,  \cite{PR5}). 
\vskip1mm
\noindent
{\bf Norm similarities of Albert algebras :} Let $A$ be an Albert algebra over $k$ and $N$ denote its norm map. A {\bf norm similarity} of $A$ is a bijective linear map $f:A\rightarrow A$ such that $N(f(x))=\nu(f)N(x)$ for all $x\in A$ and some $\nu(f)\in k^*$. Since we are working over $k$ infinite, norm similarities are synonimous with isotopies of Albert algebras (see \cite{J1}, Chap. VI, Thm. 6, Thm. 7). 
\vskip1mm
\noindent
{\bf Adjoints and $U$-operators :} Let $A$ be an Albert algebra over $k$ and $N$ be its norm map. Recall, for $a\in A$ the $U$-operator $U_a$ is given by $U_a(y):=T(a,y)a-a^{\#}\times y,~y\in A$. When $a\in A$ is invertible, it can be shown  that $U_a$ is a norm similarity of $A$, in fact, for any $x\in A,~N(U_a(x))=N(a)^2N(x)$ when $a$ is invertible. 

For a central simple algebra $D$ of degree $3$ over $k$, the adjoint map $a\mapsto a^{\#}$ satisfies $N_D(a)=aa^{\#}=a^{\#}a,~\forall~a\in D$. It can be shown that $(xy)^{\#}=y^{\#}x^{\#}$ for all $x,y\in D$. It also follows that $N(x^{\#})=N(x)^2$. 
\vskip1mm
\noindent
{\bf Algebraic groups from Albert algebras :} In this paper, for a $k$-algebra $X$ and a field extension $L$ of $k$, We will denote by $X_L$ the $L$-algebra $X\otimes_kL$. Let $A$ be an Albert algebra over $k$ with norm $N$ and $\overline{k}$ be an algebraic closure of $k$. The full group of automorphisms 
{\bf Aut}$(A):=Aut(A_{\overline{k}})$ is a simple algebraic group of type $F_4$ defined over $k$ and all simple groups of type $F_4$ defined over $k$ arise this way. When $A$ is a division algebra, {\bf Aut}$(A)$ is anisotropic over $k$. We will denote the group of $k$-rational points of {\bf Aut}$(A)$ by $Aut(A)$. Note that $A$ is a division algebra precisely when the norm form $N$ of $A$ is anisotropic (see \cite{Spr}, Thm. 17.6.5). 

The full group {\bf Str}$(A)$ of norm similarities of $N$, called the {\bf structure group} of $A$, is a connected reductive group over $k$, of type $E_6$. We denote by $Str(A)$ the group of $k$-rational points {\bf Str}$(A)(k)$. The $U$-operators $U_a$, for $a\in A$ invertible, are norm similarities and generate a normal subgroup of $Str(A)$, called the 
{\bf inner structure group} of $A$, which we will denote by $Instr(A)$.
 The commutator subgroup {\bf Isom}$(A)$ of {\bf Str}$(A)$ is the full group of isometries of $N$ and is a simple, simply connected group of type $E_6$, anisotropic over $k$ if and only if $N$ is anisotropic, if and only if $A$ is a division algebra (see \cite{SV}, \cite{Spr}). Note that {\bf Aut}$(A)$ is the stabilizer of $1\in A$ in {\bf Str}$(A)$.

On occasions, we will need to treat an algebra $X$ as an algebraic variety (particularly when dealing with $R$-equivalence etc.). In such a situation, if no confusion is likely, we shall continue to denote the underlying (affine) space by $X$. Morphisms $X\rightarrow Y$ or rational maps would carry similar meanings. Base change of an object $X$ defined over a base field $k$ to an extension $L$ of $k$ will be denoted by $X_L$. 
We record the following observations with proofs :
\begin{proposition}\label{deg-6} Let $A$ be an Albert division algebra over $k$. Then {\bf Aut}$(A)$, {\bf Str}$(A)$ and {\bf Isom}$(A)$ are 
split by a degree $6$ extension of $k$. 
\begin{enumerate}
\item $A$ is a first construction if and only if $A$ ( hence {\bf Aut}$(A)$ ) has a degree $3$ splitting field and, $A$ is a pure first construction if and only if every minimal splitting field of $A$ is cubic cyclic.
\item $A$ is a pure second construction if and only if $A$ ( hence {\bf Aut}$(A)$ ) has no splitting field of degree $3$ over $k$. 
\end{enumerate}  
\end{proposition}
\begin{proof} Write $A$ as a second construction $A=J(B,\sigma,u,\mu)$ for suitable parameters as decribed in the preliminaries. If $A$ is a first construction, any cubic subfield of $A$ splits $A$ (see \cite{PR2}, Thm. 4.7), hence splits {\bf Aut}$(A)$ and the statement follows. Let $K$ be the center of $B$, which must be a quadratic field extension of $k$ if $A$ is not a first construction. Base changing to $K$ we have $A\otimes K\cong J(B,\mu)$. By the argument presented, there is a cubic extension $L$ of $K$ that splits $A\otimes K$. Hence $A$ is split by the degree $6$ extension $LK$ of $k$. 
To prove ($1$), note that if $A$ is a first construction, $A$ has a cubic splitting field. Conversely, if $A$ has a degree $3$ splitting field $L$ over $k$, then the invariant $f_3(A)\in H^3(k,\mathbb{Z}/2)$ vanishes over $L$ and hence over $k$ by Springer's theorem. Hence $f_3(A)=0$ and $A$ is a first construction (\cite{KMRT}, 40.5). The assertion about pure first construction follows from (\cite{Th-1}, Thm. 2.1). Proof of ($2$) follows from ($1$).
\end{proof}
\begin{proposition}\label{first-tits} Let $A$ be an Albert division algebra over $k$. Then $A$ is a first Tits construction if and only if {\bf Str}$(A)$ (equivalently {\bf Isom}$(A)$) is split by a cubic extension of $k$.
\end{proposition}
\begin{proof} First assume that $A$ is an Albert division algebra over $k$ arising from a first Tits construction. Then $A$ is split by any cubic extension $L\subset A$ of $k$ (see \cite{PR2}, Thm. 4.7) and hence, by (\cite{Spr}, Thm. 17.6.3), $L$ also splits the structure group {\bf Str}$(A)$ of $A$. Conversely, if {\bf Str}$(A)$ splits over a cubic extension $L$ of $k$, then, by (\cite{Spr}, Thm. 17.6.5), the Albert algebra $A$ is reduced over $L$. By (\cite{Spr}, Thm. 17.6.3), the stabilizer of $1\in A\otimes_k L$ is the split group of type $F_4$ over $L$ and equals {\bf Aut}$(A\otimes_k L)$. Hence, by (\cite{KMRT}, 40.5), {\bf Aut}$(A\otimes_k L)$ has its mod-$2$ invariant $f_3(A\otimes_k L)=f_3(A)\otimes_k L= 0$. It follows that the coordinatizing octonion algebra of $A\otimes_k L$ is split over $L$. Hence, by Springer's theorem, $f_3(A)=0$. Therefore $A$ must be a first Tits construction over $k$ (\cite{KMRT}, 40.5).    
\end{proof}
\vskip1mm
\noindent
{\bf Strongly inner group of type $E_6$ :} Let $H$ be a simple, simply connected algebraic group defined and strongly inner of type $E_6$ over $k$ (see \cite{T1} for definition). Then $H$ is isomorphic to {\bf Isom}$(A)$, the (algebraic) group of all norm isometries of an Albert algebra $A$ over $k$. Moreover, $H$ is anisotropic if and only if $A$ is a division algebra (see \cite{Spr}, Thm. 17.6.5).  
\vskip1mm
\noindent
{\bf Groups with Tits index $E_{8,2}^{78}$ :} Let $G$ be a simple, simply connected algebraic group with Tits index $E_{8,2}^{78}$ over $k$. The index has associated diagram as below: 
\vskip10mm

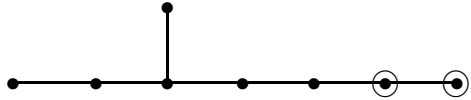
\begin{center}
\setlength{\unitlength}{1.0cm}
\begin{picture}(10,-25)\thicklines

\put(0,0){\line(1,0){1}}
\put(1,0){\line(1,0){1}}
\put(2.10,0){\line(0,1){1}}
\put(2,0){\line(1,0){1}}
\put(3,0){\line(1,0){1}}
\put(4,0){\line(1,0){1}}
\put(-0.05,-0.10){$\bullet$}
\put(1.05,-0.10){$\bullet$}
\put(2.00,-0.10){$\bullet$}
\put(3.0,-0.10){$\bullet$}
\put(3.95,-0.10){$\bullet$}
\put(2.00,0.90){$\bullet$}
\put(4.90,-0.10){$\bullet$}
\put(4.92,0){\line(1,0){1}}
\put(4.8,-0.10){$\bigcirc$}
\put(5.85,-0.10){$\bullet$}
\put(5.74,-0.10){$\bigcirc$}
\end{picture}
\end{center}
The anisotropic kernel of such a group is a simple, simply connected, strongly inner group of type $E_6$ and by (\cite{TW}, \cite{T1}), can be realized as the group {\bf Isom}$(A)$ of norm isometries  of an Albert division algebra $A$ over $k$. Let $\Gamma$ denote the Tits building associated to $G$. Then $\Gamma$ is a Moufang hexagon associated to a hexagonal system of type $27/F$, where $F$ is a quadratic field extension $K$ of $k$ (this is when $A$ is written as a second construction $J(B,\sigma,u,\mu)$ with $Z(B)=K$) or $F=k$ (when $A$ is a first construction). The group $G(k)$ is the group of ``linear'' automorphisms of $\Gamma$. The explicit realization of groups of type $E_{8,2}^{78}$ as automorphisms of a Moufang hexagon as above is due to Tits and Weiss and is detailed in their monograph (\cite{TW}).
\vskip1mm
\noindent
{\bf Groups with Tits index $E_{7,1}^{78}$ :} Let $H$ be a simple, simply connected algebraic group over $k$ with Tits index $E_{7,1}^{78}$ over $k$. Then the anisotropic kernel of $H$ is a simple, simply connected strongly inner group of type $E_6$ and is isomorphic to {\bf Isom}$(A)$ for an Albert division algebra $A$ over $k$ (see \cite{T1}). The Tits index of $H$ has associated diagram as below: 
\vskip10mm

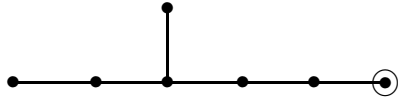
\begin{center}
\setlength{\unitlength}{1.0cm}
\begin{picture}(10,-25)\thicklines

\put(0,0){\line(1,0){1}}
\put(1,0){\line(1,0){1}}
\put(2.10,0){\line(0,1){1}}
\put(2,0){\line(1,0){1}}
\put(3,0){\line(1,0){1}}
\put(4,0){\line(1,0){1}}
\put(-0.05,-0.09){$\bullet$}
\put(1.05,-0.09){$\bullet$}
\put(2.00,-0.09){$\bullet$}
\put(3.0,-0.09){$\bullet$}
\put(3.95,-0.09){$\bullet$}
\put(2.00,0.90){$\bullet$}
\put(4.90,-0.10){$\bullet$}
\put(4.80,-0.10){$\bigcirc$}

\end{picture}
\end{center}
Following Max Koecher (\cite{Ko}), one can realize $H$ as follows: fix an Albert divsion algebra $A$ over $k$. Let ${\bf \Xi}(A)$ be the (algebraic) group of birational transformations of $A$ 
generated by the structure group ${\bf Str}(A)$ and the maps 
$$x\mapsto t_a(x)=x+a~(~a\in~A_{\overline{k}}~),~x\mapsto j(x)=-x^{-1}~(x\in~A_{\overline{k}}~).$$
Then ${\bf \Xi}(A)$ has type $E_{7,1}^{78}$ and any algebraic group over $k$ with this Tits index arises this way (see \cite{T1}). Any element of $f\in{\bf\Xi}(A)$ is of the form 
$$f=w\circ t_a\circ j\circ t_b\circ j\circ t_c,~~w\in\text{\bf Str}(A),~a,b,c\in~A_{\overline{k}}.$$ 
This was proved by Koecher when $k$ is infinite and has characteristic $\neq 2$ and by O. Loos in arbitrary chatacteristics (\cite{L}), in the context of Jordan pairs.\\
\vskip1mm
\noindent
{\bf Remark :} In light of Proposition \ref{deg-6} and Proposition \ref{first-tits}, the anisotropic kernel of a simple, simply connected group $G$ defined over $k$, with Tits index $E^{78}_{8,2}$ or $E^{78}_{7,1}$, is split by a degree $6$ extension of $k$.  

We have
\begin{corollary}\label{an-kernel} Let $G$ be a simple, simply connected algebraic group over $k$ with index $E^{78}_{8,2}$ or $E^{78}_{7,1}$, such that the anisotropic kernel $H$ of $G$ is split by a cubic extension of $k$. Then 
$H\cong\text{\bf Isom}(A)$ for an Albert division algebra $A$ over $k$ arising from the first Tits construction. 
\end{corollary}
\begin{proof} This is immediate from Proposition \ref{first-tits} and the above discussion. 
\end{proof} 
\vskip1mm
\noindent
{\bf $R$-equivalence in algebraic groups :}
Let $X$ be an irreducible variety over a field $k$ with $X(k)$ nonempty. Following Manin (\cite{M}) we define points $x,y\in X(k)$ to be $R$-equivalent 
if there exists a sequence $x_0=x, x_1,\cdots, x_n=y$ of points in $X(k)$ and rational mape $f_i:\mathbb{A}_k^1\rightarrow X,~1\leq i\leq n$, defined over $k$ and regular at $0$ and $1$, such that $f_i(0)=x_{i-1},~f_i(1)=x_i$. 

When $G$ is a connected algebraic group defined over $k$, the set of elements of $G(k)$ which are $R$-equivalent to $1\in G(k)$ is a normal subgroup $RG(k)$ of $G(k)$. The set $G(k)/R$ of $R$-equivalence classes in $G(k)$ is in natural 
bijection with the quotient $G(k)/RG(k)$ and therefore carries a natural group structure, and we identify $G(k)/R$ with the group $G(k)/RG(k)$. This group measures the obstruction to rational parameterizing of points of $G(k)$. 

We say $G$ is {\bf $R$-trivial} if $G(L)/R=\{1\}$ for all field extensions $L$ of $k$. Recall that a variety $X$ defined over $k$ is said to be $k$-{\bf rational} if $X$ is birationally isomorphic over $k$ to an affine space. It is well known (see \cite{Vos}, Chap. 6, Prop. 2) that if $G$ is $k$-rational then $G$ is $R$-trivial.  

Let $G$ be a connected reductive group defined over $k$ and assume $G$ is $k$-{\bf isotropic}, i.e., admits a $k$-embedding $\mathbb{G}_m\hookrightarrow G$ such that the image is not contained in $Z(G)$, the center of $G$. Let $G(k)^{\dagger}$ denote the subgroup of $G(k)$ generated by the $k$-rational points of the unipotent radicals of the $k$-parabolic subgroups of $G$. Then $G(k)^{\dagger}$ is normal in $G(k)$ and the quotient $W(k,G)=G(k)/G(k)^{\dagger}$ is the {\bf Whitehead group} of $G$. 

By a theorem of Gille (\cite {G}, Thm. 7.2) if $G$ is a semisimple, simply connected and absolutely almost simple group, defined and isotropic over $k$, then $W(k,G)\cong G(k)/R$. The Kneser-Tits problem asks if for such a group $G$, its Whitehead group is trivial, or more generally, to compute the Whitehead group.
In this paper, we prove 
\begin{theorem}Let $G$ be a simple, simply connected algebraic group over $k$ of type $E_{7,1}^{78}$ or $E_{8,2}^{78}$ or strongly inner type $E_6$, whose anisotropic kernel is split by a cubic extension of $k$. Then $G$ is $R$-trivial.
\end{theorem}
\begin{corollary} Let $G$ be a simple, simply connected algebraic group, defined and isotropic over a field $k$, with Tits index $E^{78}_{8,2}$ or $E^{78}_{7,1}$, and anisotropic kernel split by a cubic extension of $k$. then the Kneser-Tits conjecture holds for $G$. Moreover, $G(L)$ is projectively simple, for any field extension $L$ or $k$. 
\end{corollary}
\begin{corollary} Let $G$ be a simple, simply connected algebraic group, defined and isotropic over a field $k$, with Tits index $E^{78}_{8,2}$ or $E^{78}_{7,1}$. Then there exists a quadratic field extension $K$ of $k$ such that $G$ is $R$-trivial over $K$, i.e., for any extension $F$ of $K$, $G(F)/R=\{1\}$.
\end{corollary}
\noindent
We will prove this result in steps. The key result in proving the above theorem is our {\bf Main theorem}: 
\begin{theorem}
Let $A$ be an Albert division algebra over $k$ arising from the first Tits construction. Then {\bf Str}$(A)$ is $R$-trivial. 
\end{theorem}
\begin{corollary}{\bf Tits-Weiss conjecture :} Let $A$ be an Albert division algebra arising from the first Tits construction. Then $Str(A)=C.Instr(A)$, where $C$ is the subgroup of $Str(A)$ consisting of scalar homotheties $R_a:A\rightarrow A,~a\in k^*$.
\end{corollary}
\section{\bf Automorphisms and $R$-triviality}
Let $A$ be an Albert algebra over $k$ and $S$ a subalgebra of $A$. We will denote by {\bf Aut}$(A/S)$ the (closed) subgroup of {\bf Aut}$(A)$ consisting of automorphisms of $A$ which fix $S$ pointwise and {\bf Aut}$(A,S)$ will denote the closed subgroup of automorphisms of $A$ stabilizing $S$. The group of $k$-rational points of these groups will be denoted by ordinary fonts, for example $Aut(A)=$ {\bf Aut}$(A)(k)$ and so on. When $A$ is a division algebra, any proper subalgebra $S$ of $A$ is either $k$ or a cubic subfield or a $9$-dimensional (degree $3$) Jordan subalgebra, which is of the form $D_+$ for a degree $3$ central division algebra $D$ over $k$ or of the form $(B,\sigma)_+$ for a central division algebra $B$ of degree $3$ over a quadratic extension $K$ of $k$, with an involution $\sigma$ of the second kind over $K/k$ (cf. Thm. 1.1, \cite{PR6}). We record below a description of some subgroups of {\bf Aut}$(A)$ which we shall use in the sequel, we refer to (\cite{KMRT}, 39.B) and (\cite{FP}) for details.
\begin{proposition}\label{rational} Let $A$ be an Albert algebra over $k$.
\begin{enumerate} 
\item Suppose $S=D_+\subset A$ for $D$ a degree $3$ central simple $k$-algebra. Write $A=J(D,\mu)$ for some $\mu\in k^*$. Then {\bf Aut}$(A/S)\cong$ {\bf SL}$(1,D)$, the algebraic group of norm $1$ elements in $D$. 
\item Let $S=(B,\sigma)_+\subset A$ for $B$ a degree $3$ central simple algebra of degree $3$ over a quadratic extension $K$ of $k$ with an involution $\sigma$ of the second kind over $K/k$. Write $A=J(B,\sigma,u,\mu)$ for suitable parameters. Then 
{\bf Aut}$(A/S) \cong $ {\bf SU}$(B,\sigma_u)$, where $\sigma_u:=Int(u)\circ\sigma$. 
\end{enumerate}
In particular, the subgroups described in (1) and (2) are simply connected, simple of type $A_2$, defined over $k$ and hence are rational, therefore are $R$-trivial. 
\end{proposition}
\noindent
The last assertion is due to the fact that connected reductive groups of absolute rank at most $2$ are rational (see \cite{CP}). For an algebraic group $G$ defined over $k$, we denote the (normal) subgroup of elements in $G(k)$ that are $R$-equivalent to $1\in G(k)$ by $R(G(k))$.\\
\vskip1mm
\noindent
{\bf Extending automorphisms :} Let $A$ be an Albert (division) algebra and $S\subset A$ a subalgebra. Then, under certain conditions, we can extend automorphisms of $S$ to automorphisms of $A$ (e.g. when $S$ is nine dimensional, see \cite{P-S-T1}, Thm. 3.1, \cite{P}, Thm. 5.2). We need certain explicit extensions for our results, we proceed to describe these (cf. \cite{KMRT}, 39.B). 
We will need the following important result in the sequel frequently, so we record it below (\cite{Me} or \cite{KMRT}, 40.13 and Lemma 5.4 of \cite{P} for arbitrary characteristics):
\begin{lemma}\label{unitary} Let $(B,\sigma)$ be a degree $3$ central simple algebra with center $K$, a quadratic \'{e}tale algebra and $\sigma$ an involution of the second kind. Let $w\in B^*$ be such that $\lambda:=N_B(w)$ satisfies $\lambda\overline{\lambda}=1$. Then there exists $q\in U(B,\sigma)$ such that $\lambda=N_B(q)$. 
\end{lemma}

First, let $S$ be a $9$-dimensional subalgebra of $A$. Then we may assume $S=(B,\sigma)_+$ for a degree $3$ central simple algebra with center $K$, 
a quadratic \'{e}tale extension of $k$ (see \cite{PR6}). 

We then have 
$A\cong J(B,\sigma, u,\mu)$ for suitable parameters. 
Recall that any automorphism of $S$ is of the form $x\mapsto pxp^{-1}$ with $p\in Sim(B,\sigma)$, where $Sim(B,\sigma)=\{g\in B^*|g\sigma(g)\in k^*\}$. This is immediate from the fact that any norm similarity of $S$ is of the form $x\mapsto \gamma gx\sigma(g)$ for some $\gamma\in k^*$ and $g\in B^*$ (see Thm. 5.12.1, \cite{J3}) and that among these, automorphisms are precisely the norm similarities that fix the identity element of $S$ (see \cite{J-4}, Thm. 4). Recall the notation $\sigma_u:=Int(u)\circ\sigma$. We have,
\begin{proposition}\label{aut-ext} Let $A=J(B,\sigma, u,\mu)$ be an Albert algebra over $k$, $S:=(B,\sigma)_+$ and $K=Z(B)$ be as above. Let $\phi\in Aut(S)$ be given by $\phi(x)=gxg^{-1}$ for 
$g\in Sim(B,\sigma)$ with $g\sigma(g)=\lambda\in k^*$ and $\nu:=N_B(g)\in K^*$. Let $q\in U(B,\sigma_u)$ be arbitrary with $N_B(q)=\bar{\nu}^{-1}\nu$. Then the map $\widetilde{\phi}:A\rightarrow A$, given by 
$$(a,b)\mapsto (gag^{-1},\lambda^{-1}\sigma(g)^{\#}bq),$$ 
is an automorphism of $A$ extending $\phi$.  
\end{proposition}
\begin{proof} To show $\widetilde{\phi}$ is an automorphism, we first note that $\widetilde{\phi}$ is a $k$-linear bijection. Hence, by (\cite{McK}, Page 507), it suffices to check $\widetilde{\phi}(1)=1$ and 
$\widetilde{\phi}$ preserves norms. The first condition is obvious. Let us verify the second. We first note some formulae which will be used in the computation below: 
$$N_B(\sigma(g))=\overline{\nu},~N_B(g^{\#})=N_B(g)^2=\nu^2,~N_B(\sigma(g)^{\#})=\overline{\nu}^2,~N_B(g\sigma(g))=\nu\overline{\nu}=\lambda^3,$$
$$N_B(g)=gg^{\#}=\nu,~\lambda^2=(g\sigma(g))^{\#}=\sigma(g)^{\#}g^{\#},~N_B(q)=\overline{\nu}^{-1}\nu.$$
Also, since $q\in U(B,\sigma_u)$, we have
$$1=q\sigma_u(q)=qu\sigma(q)u^{-1},~~qu=u\sigma(q)^{-1}.$$
Now, for $(a,b)\in A$, we compute using the above formulae,
$$N(\widetilde{\phi}(a,b))=N((gag^{-1},\lambda^{-1}\sigma(g)^{\#}bq))~~~~~~~~~~~~~~~~~~~~~~~~~~~~~~~~~~~~~~~~~~~~~~~~~~~~~~~~~~~~~~~~~~~~~~$$
$$=N_B(gag^{-1})+T_K(\mu N_B(\lambda^{-1}\sigma(g)^{\#}bq))-T_B(gag^{-1}(\lambda^{-1}\sigma(g)^{\#}bq)u\sigma(\lambda^{-1}\sigma(g)^{\#}bq))~~~~~~~~~~~~~~~~~~~~~~~~~~~~~~~~~~~~~$$
$$=N_B(a)+T_K(\mu \lambda^{-3}\bar{\nu}^2N_B(b)N_B(q))-T_B(gag^{-1}(\lambda^{-1}\sigma(g)^{\#}bq)u(\sigma(q)\sigma(b) g^{\#}\lambda^{-1}))~~~~~~~~~~~~~~~~~~~~~~~~~~~~~~~~~~~~~$$
$$=N_B(a)+T_K(\mu(\nu\bar{\nu})^{-1}\bar{\nu}^2 N_B(b)\bar{\nu}^{-1}\nu)-T_B(gag^{-1}\lambda^{-1}\lambda^2(g^{\#})^{-1}bqu\sigma(q)\sigma(b)\nu g^{-1}\lambda^{-1})~~~~~~~~~~~~~~~~~~~~~~~~~~~~~~~$$
$$=N_B(a)+T_K(\mu N_B(b))-T_B(gag^{-1}g\nu^{-1} b q u\sigma(q)\sigma(b)\nu g^{-1})~~~~~~~~~~~~~~~~~~~~~~~~~~~~~~~~~~~~~~~~~~~~~~~~~~~~~~~~~~~~~~~~~~~~~~~~$$ 
$$=N_B(a)+T_K(\mu N_B(b))-T_B(a b u\sigma(b))=N((a,b)).~~~~~~~~~~~~~~~~~~~~~~~~~~~~~~~~~~~~~~~~~~~~~~~~~~~~~~~~~~~~~~~~~~~~~~~~~~~~~~~~~~~~~~~~~~~~~~~~~~$$
Hence $\widetilde{\phi}$ is an automorphism of $A$ extending $\phi$.
\end{proof}
\begin{corollary}\label{first-aut-ext} Let $A=J(D,\mu)$ be an Albert algebra, which is a first construction. Let $\phi\in Aut(D_+)$ be given by $\phi(x)=gxg^{-1}$ for $g\in D^*$ and $x\in D_+$. Then the map 
$$(x,y,z)\mapsto (gxg^{-1}, gyh^{-1}, hzg^{-1}), $$
for any $h\in D^*$ with $N_D(g)=N_D(h)$, is an automorphism of $A$ extending $\phi$. 
\end{corollary}
\begin{proof} We identify $A=J(D,\mu)$ with the second construction $J(D\times D^{\circ}, \epsilon, 1, (\mu,\mu^{-1}))$, where $\epsilon$ is the switch involution on $B=D\times D^{\circ}$ and apply the proposition.  
\end{proof}
\noindent
Note that, for $g\in Sim(B,\sigma)$, the extension $\widetilde{\phi}$ of $\phi$ as constructed above, can be simplified as follows: let $p=\lambda\nu^{-1} g$, where $g\sigma(g)=\lambda\in k^*$ and $\nu=N_B(g)\in K^*$. Then $\lambda^3=\nu\bar{\nu}$ and $p\in U(B,\sigma)$, since 
$$p\sigma(p)=\lambda\nu^{-1}g\lambda\bar{\nu}^{-1}\sigma(g)=\lambda^2(\nu\bar{\nu})^{-1}(g\sigma(g))=\lambda^3\lambda^{-3}=1.$$
Hence the expression for $\widetilde{\phi}$, using $g\sigma(g)=\lambda,~\sigma(g^{\#})g^{\#}=\lambda^2$, $g g^{\#}=\nu$, becomes,
$$\widetilde{\phi}(a,b)=(gag^{-1},\lambda^{-1}\sigma(g^{\#})bq)$$
$$~~~~~~~~~~~~~~~~~~~~~~~=(\lambda\nu^{-1}ga\lambda^{-1}\nu g^{-1},\lambda^{-1}\bar{\nu}\sigma(g)^{-1}bq)$$
$$=(pap^{-1}, \sigma(\lambda^{-1}\nu g^{-1})bq)=(pap^{-1},\sigma(p)^{-1}bq)=(pap^{-1}, pbq).~~~~~~~$$
\noindent
We have,
\begin{proposition} Let $A=J(B,\sigma, u, \mu)$ be an Albert algebra and $S:=(B,\sigma)_+\subset A$ and $K=Z(B)$, a quadratic \'{e}tale extension of $k$. Then any automorphism of $A$ stabilizing $S$ is of the form $(a,b)\mapsto (pap^{-1}, pbq)$ for $p\in U(B,\sigma)$ and $q\in U(B,\sigma_u)$ with 
$N_B(p)N_B(q)=1$. We have 
$$Aut(A, (B,\sigma)_+)\cong [U(B,\sigma)\times U(B,\sigma_u)]^{det}/K^{(1)},$$
where
$$[U(B,\sigma)\times U(B,\sigma_u)]^{det}:=\{(p,q)\in U(B,\sigma)\times U(B,\sigma_u)| N_B(p)=N_B(q)\},$$
and $K^{(1)}$ denotes the group of norm $1$ elements in $K$, embedded diagonally.
\end{proposition}
\begin{proof} Let $\phi\in Aut(A,S)$ and let $\eta:=\phi|S\in Aut(S)$. Then $\eta(a)=gag^{-1}$ for some $g\in Sim(B,\sigma)$. Let $\lambda:=g\sigma(g)\in k^*$ and $\nu:=N_B(g)\in K^*$. We extend $\eta$ to an automorphism $\widetilde{\eta}$ of $A$ by the formula 
$$\widetilde{\eta}(a,b)=(pap^{-1}, pbq'),$$
where $p=\lambda\nu^{-1} g$ and $q'\in U(B,\sigma_u)$ is arbitrary with $N_B(q')N_B(p)=1$. Then $p\in U(B,\sigma)$ as discussed above and the automorphism 
$\widetilde{\eta}^{-1}\phi$ of $A$ fixes $S=(B,\sigma)_+$ pointwise. Hence, by (\cite{KMRT}, 39.16), 
$$\widetilde{\eta}^{-1}\phi(a,b)=(a,bq'')$$
for some $q''\in SU(B,\sigma_u)$. Therefore we get 
$$\phi(a,b)=\widetilde{\eta}(a,bq'')=(pap^{-1},pbq''q').$$ 
We set $q=q''q'\in U(B,\sigma_u)$ and the proof of the first assertion is complete. Now, if $q\in U(B,\sigma_u)$ is arbitrary, with $N_B(q)=\nu\in K^*$. Then $\nu\overline{\nu}=1$. Hence, by Lemma \ref{unitary}, there exists $p\in U(B,\sigma)$ with $N_B(p)=\nu^{-1}$. 

Then, by the computations done before, it follows that $(b,x)\mapsto (pbp^{-1}, pxq)$ is an automorphism of $A$ stabilizing $(B,\sigma)_+$. Hence any automorphism $\phi$ of $A$ which stabilizes $(B,\sigma)_+$ is precisely of the form $(b,x)\mapsto (pbp^{-1}, pxq)$ for some $p\in U(B,\sigma),~q\in U(B,\sigma_u)$ with $N_B(p)N_B(q)=1$. 

The map $f:[U(B,\sigma)\times U(B,\sigma_u)]^{det}\rightarrow Aut(A,(B,\sigma)_+)$ given by $(p,q)\mapsto [(b,x)\mapsto (pbp^{-1}, pxq^{-1})]$ is a surjective homomorphism with 
$ker(f)= \{(p,p)|p\in K^{(1)}\}\cong K^{(1)}$, inducing the required isomorphism between the two groups.  
\end{proof}
\begin{corollary}\label{aut-ext-first} Let $A=J(D,\mu)$ be a first Tits construction Albert algebra. Then any automorphism of $A$ stabilizing $D_+\subset A$ is of the form 
$$(x,y,z)\mapsto (gxg^{-1},gyh^{-1}, hzg^{-1}),$$
where $g, h\in D^*$ and $N_D(h)=N_D(g)$. 
\end{corollary}
\begin{proof} As before, we identify $A=J(D,\mu)$ with a second Tits construction $J(D\times D^{\circ},\epsilon, 1, (\mu,\mu^{-1}))$ and apply the proposition.
\end{proof}
\section{\bf Norm similarities and $R$-triviality}
In the last section, we extended automorphisms of a $9$-dimensional subalgebra of an Albert algebra to an automorphism of the Albert algebra. In this section, we analyse norm similarities of an Albert algebra $A$ in the same spirit, study rationality and $R$-triviality properties of some subgroups of {\bf Str}$(A)$. 

First we make some simple observations. Since any norm similarity of $A$ fixing the identity element $1$ of $A$ is necessarily an automorphism of $A$, it follows that for any $k$-subalgebra $S\subset A$ the subgroup {\bf Str}$(A/S)$ of all norm similarities that fix $S$ pointwise, is equal to the subgroup {\bf Aut}$(A/S)$ of {\bf Aut}$(A)$, i.e., {\bf Str}$(A/S)=${\bf Aut}$(A/S)$. Recall that we denote the (normal) subgroup of $Str(A)$ generated by the $U$-operators of $A$, by $Instr(A)$. We will now prove some results on norm similarities that will be needed in the paper. We begin with
\begin{theorem}\label{Instr} Let $A$ be any Albert algebra. Then $Instr(A)\subset R(${\bf Str}$(A)(k))$.
\end{theorem}
\begin{proof} Let $\psi\in Instr(A)$. Then $\psi=U_aU_bU_c\cdots$ with $a,b,c,\cdots \in A$. Let, for $t\in\mathbb{A}^1_k,~a_t:=(1-t)a+t\in A_{\overline{k}}$, similarly $b_t, c_t,\cdots $ be defined. Let 
$\theta:\mathbb{A}_k^1\rightarrow${\bf Str}$(A)$ be defined by $\theta(t)=U_{a_t}U_{b_t}U_{c_t}\cdots$ Then $\theta$ is a rational map defined over $k$, with domain the (open) subset of $\mathbb{A}_k^1 $ consisting of $t$ such that $a_t, b_t, c_t, \cdots$ are invertible elements of $A_{\bar{k}}$. Clearly $\theta$ is regular and defined at $0$ and $1$ and $\theta(t)\in~${\bf Str}$(A)$ with $\theta(0)=\psi,~\theta(1)=1$. Hence the assertion is proved. 
\end{proof}
\noindent
The theorem below is a generalization of Corollary 4.2 of (\cite{Th-1}) over a field $k$ of arbitrary characteristic :
\begin{theorem}\label{fixedpoint} Let $A$ be an Albert division algebra over $k$ and $\phi\in Aut(A)$ an automorphism of $A$. Then $\phi$ fixes a cubic subfield of $A$ pointwise.
\end{theorem}
\begin{proof} Let $A^{\phi}$ denote the subalgebra of fixed points of $\phi$. In light of the fact that proper subalgebras of an Albert division algebras are of dimension $1$, $3$ or $9$, 
and that $Dim_k(A^{\phi})$ is invariant under base change, we may work with $A_{\overline{k}}$. Since the eigenvalues of $\phi$ are the same as the eigenvalues of 
its semisimple part, we may further assume $\phi$ is semisimple. Let $T$ be a maximal torus containing $\phi$. By (Theorem 3, \cite{F1}), the group of all automorphisms of $A_{\overline{k}}$ fixing a primitive idempotent is isomorphic to $Spin(9)$. We can conjugate $T$ to a maximal torus in $Spin(9)$ (over $\overline{k}$) to conclude that $T$ fixes a vector $w,~w\notin\overline{k}$. Hence $\phi$ fixes some $v\in A,~v\notin k$. Then $v$ generates a cubic subfield $L$ of $A$ and $\phi$ fixes $L$ pointwise. 
\end{proof}
\noindent
The following theorem is a generalization of Theorem 5.3 of (\cite{Th-1}) to arbitrary characteristics :
\begin{theorem}\label{cyclic-auto} Let $A$ be a first Tits construction Albert division algebra over $k$. Let $\phi\in Aut(A)$ be such that $\phi$ fixes a cubic cyclic subfield of $A$ pointwise. Then $\phi$ is a product of two automorphisms, each stabilizing a $9$-dimensional subalgebra of $A$.
\end{theorem} 
\begin{proof} The proof goes along the exact lines of the proof given in (\cite{Th-1}), except that we use Theorem \ref{fixedpoint} above in place of (Corollary 4.2, \cite{Th-1}).
\end{proof}
We have the following corollary to the above, with proof exactly as in (\cite{Th-1}, Cor. 5.1), we will use this in the paper for $k$ of arbitrary characteristics:
\begin{corollary}\label{cubic-cyclic} Let $A$ be a first Tits construction Albert division algebra over $k$ and $\phi\in Aut(A)$ be such that $\phi(L)=L$ for a cubic cyclic subfield $L\subset A$. Then $\phi$ is a product of three automorphisms of $A$, each stabilizing a $9$-dimensional subalgebra of $A$.
\end{corollary}
\noindent
The following result is a generalization of Theorem 5.10 of (\cite{Th-1}), for arbitrary characteristics : 
\begin{theorem}\label{str-gate} Let $A$ be an Albert division algebra arising from the first Tits construction. Let $L\subset A$ be a cubic cyclic subfield. Then $Str(A,L)\subset C. Instr(A).H$, where $H$ denotes the subgroup of $Aut(A)$ generated by automorphisms stabilizing $9$-dimensional subalgebras of $A$ and $C$ is the group of all scalar homotheties in $Str(A)$ .
\end{theorem}
\begin{proof} The proof goes along the exact lines of the proof of Thm 5.10 in (\cite{Th-1}), except that we use Corollary \ref{cubic-cyclic} above instead of (Cor.5.1, \cite{Th-1}).
\end{proof}
\noindent
{\bf Extending norm similarities :} Let $A$ be an Albert (division) algebra and $S$ a $9$-dimensional subalgebra of $A$. 
Given an element $\psi\in Str(S)$, we wish to extend it to an element of $Str(A)$. Though this can be achieved using a recent result of Garibaldi-Petersson (\cite{GP}, Prop. 7.2.4), we need certain explicit extension for the purpose of proving $R$-triviality of the algebraic groups described in the introduction, which we now proceed to describe.

We have $S\cong(B,\sigma)_+$ for a suitable degree $3$ central division algebra $B$ over a quadratic extension $K$ of $k$ with an involution $\sigma$ of the second kind. Therefore we have $A\cong J(B,\sigma,u,\mu)$ for suitable parameters. For our purpose, by conjugating with a suitable automprphism of $A$, we may assume $S=(B,\sigma)_+,~A=J(B,\sigma,u,\mu)$ and $\psi\in Str((B,\sigma)_+)$. Let $b_0=\psi(1)$ and $\lambda=N_B(b_0)$. 
By (\cite{J3}, Chap.V, Thm. 5.12.10), $Str((B,\sigma)_+)$ consists of the maps 
$$x\mapsto \lambda ax\sigma(a),~x\in (B,\sigma)_+, a\in B^*, \lambda\in k^*.$$
Hence there exists $\gamma\in k^*$ and $g\in B^*$ such that 
$$\psi(x)=\gamma g x\sigma(g),~\forall x\in S.$$
Let $\nu:=N_B(g)\in K=Z(B)$ and let 
$$\delta=N_B(\sigma_u(g)^{-1}g)=\overline{N_B(g)}^{-1}N_B(g)=\bar{\nu}^{-1}\nu,$$ 
where $\sigma_u=Int(u)\circ\sigma$. Then $\delta\bar{\delta}=1$. By Lemma \ref{unitary}, there exists $q\in U(B,\sigma_u)$ such that $\delta=N_B(q)$. Define $\widetilde{\psi}:A\rightarrow A$ by setting
$$~~~~~~~~~~~~~~~~~~~~~~~~~~~~~~~~~~\widetilde{\psi}((b,x))=\gamma(gb\sigma(g),\sigma(g)^{\#}xq).~~~~~~~~~~~~~~~~~~~~~~~~~~~~~~~~~~(*)$$
Clearly $\widetilde{\psi}((b,0))=(\psi(b),0)$. Hence $\widetilde{\psi}$ is an extension of $\psi$ to $A$ and is clearly a $k$-linear bijective endomorphism of $A$. We now compute
$$N(\widetilde{\psi}((b,x))) = \gamma^3N((gb\sigma(g),\sigma(g)^{\#}xq))~~~~~~~~~~~~~~~~~~~~~~~~~~~~~~~~~~~~~~~~~~~~~~~~~~~~~~~~~~~~~~~~~~~~~~~~~~~~~$$
$$~~~~~~~~= \gamma^3[N_B(g)\overline{N_B(g)}N_B(b)+T_K(\mu \overline{N_B(g)}^2N_B(x)N_B(q))~~~~~~~~~~~~~~~~~~~$$
$$-T_B(gb\sigma(b)\sigma(g)^{\#}xqu\sigma(q)\sigma(x)g^{\#})].$$
Now substituting $N(q)=\overline{N_B(g)}^{-1}N_B(g)$, $\nu=N_B(g)=gg^{\#}$ and $qu=u\sigma(q)^{-1}$ using $q\sigma_u(q)=1=qu\sigma(q)u^{-1}$ in the RHS of the above equation and simplifying, we get 
$$N(\widetilde{\psi}((b,x)))=\gamma^3\nu\bar{\nu}N((b,x)).$$
Therefore $\widetilde{\psi}\in Str(A,S)$. We have proved
\begin{theorem}\label{ext-norm-sim} Let $A$ be an Albert division algebra over $k$ and $S\subset A$ be a $9$-dimensional subalgebra. Then every element $\psi\in Str(S)$ admits an extension $\widetilde{\psi}\in Str(A)$. Let 
$S=(B,\sigma)_+$ and $A=J(B,\sigma,u,\mu)$ for suitable parameters. Let $\psi\in Str(S)$ be given by $b\mapsto\gamma gb\sigma(g)$ for $b\in (B,\sigma)_+,~g\in B^*$ and $\gamma\in k^*$. Then for any $q\in U(B,\sigma_u)$ with $N_B(q)=N_B(\sigma(g)^{-1}g)$, the map $\widetilde{\psi}$ given by 
$$\widetilde{\psi}((b,x))=\gamma(gb\sigma(g),\sigma(g)^{\#}xq),$$
is a norm similarity of $A$ extending $\psi$. 
\end{theorem} 
\begin{corollary} Let $A=J(D,\mu)$ be a first construction Albert algebra. Let $\psi\in Str(D_+)$ be given by $x\mapsto \gamma axb$ for $\gamma\in k^*$ and $a,b\in D^*$. Then the map $\widetilde{\psi}:A\rightarrow A$ given by 
$$(x,y,z)\mapsto \gamma(axb, b^{\#}yc,c^{-1}za^{\#}),$$
where $b,c\in D^*$ are such that $N_D(a)=N_D(b)N_D(c)$, is a norm similarity of $A$ extending $\psi$.
\end{corollary}
\noindent
Now, let $\eta\in Str(A,S)$ be arbitrary and let $\phi:=\eta|S$. We extend $\phi$ to an element $\widetilde{\phi}\in Str(A)$ according to $(*)$. Then we have 
$$\widetilde{\phi}^{-1}\circ\eta\in Str(A/S)=Aut(A/S)\cong SU(B,\sigma_u).$$
Hence there exists $q'\in SU(B,\sigma_u)$ such that 
$$\widetilde{\phi}^{-1}\circ\eta((b,x))=(b,xq')~\forall b\in (B,\sigma)_+,~x\in B.$$ 
Write $\theta_{q'}$ for the automorphism $(b,x)\mapsto (b,xq')$. Then $\eta=\widetilde{\phi}\circ\theta_{q'}$. Also, since $\eta((B,\sigma)_+)=(B,\sigma)_+$, there exists $\gamma\in k^*$ and $g\in B^*$, such that 
$$\eta((b,0))=\gamma(gb\sigma(g),0)~\forall b\in (B,\sigma)_+.$$
Therefore, we have 
$$\eta((b,x))=\gamma(gb\sigma(g),\sigma(g)^{\#}xq_0)$$
for all $(b,x)\in A$ and some $q_0\in U(B,\sigma_u)$ with $N(q_0)=N_B(\sigma(g)^{-1}g)$. We have therefore associated a triple $(\gamma, (g, q_0))\in k^*\times B^*\times U(B,\sigma_u)$ to $\eta\in Str(A,S)$ with $\gamma\in k^*$ and $g\in B^*$ arbtrary and $q_0\in U(B,\sigma_u)$ with $N_B(q_0)=N_B(\sigma(g)^{-1}g)$. 

Conversely, for any triple $(\gamma,(g,q_0))\in k^*\times B^*\times U(B,\sigma_u)$ with $\gamma\in k^*,~g\in B^*$ arbitrary and $q_0\in U(B,\sigma_u)$ such that $N_B(q_0)=N_B(\sigma(g)^{-1}g)$, we define $\eta:A\rightarrow A$ by the formula $(*)$ :
$$\eta((b,x))=\gamma(gb\sigma(g), \sigma(g)^{\#}xq_0).$$
By a straight forward calculation exactly as above, it follows that 
$$N(\eta((b,x)))=\gamma^3\nu\bar{\nu}N((b,x)).$$
Therefore $\eta\in Str(A,S)$. Hence we have proved
\begin{theorem}\label{Str-ABsigma+}
Let $A=J(B,\sigma,u,\mu)$ be an Albert division algebra, written as a second Tits construction and $S=(B,\sigma)_+$. The group $Str(A,S)$ consists of the maps 
$(b,x)\mapsto\gamma(gb\sigma(b),\sigma(g)^{\#}xq)$ where $\gamma\in k^*, g\in B^*$ are arbitrary and $q\in U(B,\sigma_u)$ satisfies $N_B(q)=N_B(\sigma(g)^{-1}g)$. 
We have,
$$Str(A,(B,\sigma)_+)\cong \frac{k^*\times H_0}{K^*},$$
where $H_0=\{(g,q)\in B^*\times U(B,\sigma_u)|N_B(q)N_B(\sigma(g)^{-1}g)=1\}$ and $K^*\hookrightarrow k^*\times H_0$ via $\alpha\mapsto (N_K(\alpha)^{-1},\alpha,\overline{\alpha}^{-1}\alpha)$.
\end{theorem}
\begin{proof} Only the last assertion needs to be proved. Let 
$$H_0=\{(g,q)\in B^*\times U(B,\sigma_u)|N_B(q)N_B(\sigma(g)^{-1}g)=1\}.$$ 
We define 
$f:k^*\times H_0\rightarrow Str(A,(B,\sigma)_+)$ by 
$$(\gamma,(g,q))\mapsto [(b,x)\mapsto \gamma.(gb\sigma(g),\sigma(g)^{\#}xq^{-1})].$$
Then, by our calculations before, $f$ is a surjective homomorphism. We compute $ker(f)$. Let $(\gamma,(g,q))\in Ker(f)$. Then 
$$~~~~~~~~~~~~~~~~~~~~~~~~~\gamma(gb\sigma(g),\sigma(g)^{\#}xq^{-1})=(b,x)~\forall b\in (B,\sigma)_+,~x\in B.~~~~~~~~~~~~~~~~~~~~~(**)$$
Taking $b=1,~x=1$ in the above relation gives 
$$~~~~~~~~~~~~~~~~~~~~~~~~~~~~~~~~~~~~\gamma g\sigma(g)=1,~\gamma\sigma(g)^{\#}q^{-1}=1.~~~~~~~~~~~~~~~~~~~~~~~~~~~~~~~~~~~~~~~~~~~(***)$$
Now $g\sigma(g)=\gamma^{-1}$ gives $N_B(g)\overline{N_B(g)}=\gamma^{-3}$ and $\sigma(g)^{\#}=\gamma^{-2}(g^{\#})^{-1}$. Also $gg^{\#}=N_B(g)$ gives $(g^{\#})^{-1}=N_B(g)^{-1}g$. Therefore, we get from the above equation 
$$\gamma\sigma(g)^{\#}=\gamma(\gamma^{-2}(g^{\#})^{-1})=\gamma^{-1}N_B(g)^{-1}g=q,$$
and thus $g=\gamma N_B(g)q$ and $\sigma(g)^{\#}=\gamma^{-1}q$. Substituting this in $(**)$ we get
$$\gamma(\gamma N_B(g)qb\sigma(q)\overline{N_B(g)}\gamma, \gamma^{-1}qxq^{-1})=(b,x)~\forall b\in (B,\sigma)_+,~x\in B.$$
This gives, using $N_B(g)\overline{N_B(g)}=\gamma^{-3}$, 
$$(qbq^{-1}, qxq^{-1})=(b,x)~\forall b\in (B,\sigma)_+,~x\in B.$$
Hence $q\in K^*\cap U(B,\sigma_u)=K^{(1)}$. Thus $g=\gamma N_B(g)q\in K^*$ and, by $(***)$, $N_K(g)=g\sigma(g)=\gamma^{-1}$. Therefore, 
$$q=\gamma^{-1}N_B(g)^{-1}g=\gamma^{-1}g^{-3}g=\gamma^{-1}g^{-2}=\sigma(g)gg^{-2}=\sigma(g)g^{-1}.$$
Hence $(\gamma,g,q)=(N_K(g)^{-1},g, \sigma(g)g^{-1})$. This shows 
$$ker(f)=\{(N_K(\alpha)^{-1},\alpha,\overline{\alpha}^{-1}\alpha)|\alpha\in K^*\},$$
and the proof is complete.
\end{proof}
\begin{corollary} Let $A=J(D,\mu)$ be a first construction Albert algebra. Then the group $Str(A,D_+)$ consists of the maps 
$$(x,y,z)\mapsto \gamma(axb,b^{\#}yc, c^{-1}za^{\#}),$$
where $a,b,c\in D^*,~\gamma\in k^*$ and $N_D(a)=N_D(b)N_D(c)$.
\end{corollary}

\section{\bf $R$-triviality results} In this section, we proceed to develop our results on $R$-triviality of various groups and also prove our main theorem. 
We need a version of a theorem of Yanchevskii (\cite{KMRT}, 17.25) for degree $3$ division algebras in arbitrary characteristics, which we prove now:
\begin{proposition}\label{yanchevskii} Let $(B,\sigma)$ be a degree $3$ central division algebra over its center $K$, a quadratic \'{e}tale extension of $k$, with $\sigma$ an involution of the second kind. Then every nonzero element $g\in B$ admits a factorization $g=zs$ with $s\in (B,\sigma)_+$ and $z\in B^*$ is such that $z\sigma(z)=\sigma(z)z$.
\end{proposition}
\begin{proof} We may assume characteristic of $k$ is $2$, the proof when characteristic is not $2$ is as in (\cite{KMRT}, 17.25). Then $(B,\sigma)_+=\text{Sym}(B,\sigma)=\text{Skew}(B,\sigma)$ and 
$\text{Dim}_k(\text{Sym}(B,\sigma))=9$. Since $\sigma$ restricts to the nontrivial $k$-automorphism of $K$, $K\not\subset\text{Sym}(B,\sigma)$. 
Hence $\text{Dim}_k(K+\text{Sym}(B,\sigma))=10$. Also $g\in B^*$, hence 
$$\text{Dim}_k(g.\text{Sym}(B,\sigma))=\text{Dim}_k(\text{Sym}(B,\sigma))=9.$$ 
Hence $g.\text{Sym}(B,\sigma)\cap (K+\text{Sym}(B,\sigma))\neq 0$. Thus we can find $s_0\in\text{Sym}(B,\sigma)-\{0\}$, $\lambda\in K$ and $z_0\in\text{Sym}(B,\sigma)$ such that 
$gs_0=\lambda+z_0$. Let $z:=\lambda+z_0$. Then $z\in B^*$ and  
$$z\sigma(z)=(\lambda+z_0)(\sigma(\lambda)+z_0)=(\sigma(\lambda)+z_0)(\lambda+z_0)=\sigma(z)z.$$
Moreover, $g.s_0=z$ implies $g=z.(s_0)^{-1}$. Hence we have the required factorization $g=zs$ for $s=s_0^{-1}$.
\end{proof} 
\noindent
We now have
\begin{theorem}\label{Str-R-triv} Let $A$ be an Albert (division) algebra. Let $S\subset A$ be a $9$-dimensional subalgebra. Then, with the notations as above, {\bf Str}$(A,S)$ is $R$-trivial. 
\end{theorem}
\begin{proof} Let $\psi\in Str(A,S)$. Since $Dim(S)=9$, $S\cong(B,\sigma)_+$ for a suitable degree $3$ central simple algebra over a quadratic \'{e}tale extension $K$ of $k$, with an involution $\sigma$ of the second kind. Conjugating by a suitable automorphism of $A$, we may assume $S=(B,\sigma)_+$ and $A=J(B,\sigma, u,\mu)$ for suitable parameters. 
Let 
$$H_0=\{(g,q)\in B^*\times U(B,\sigma_u)|N_B(\sigma(g)^{-1}g)=N_B(q)\}.$$ 
Clearly $H_0$ is a subgroup of $B^*\times U(B,\sigma_u)$. Then, by Theorem \ref{Str-ABsigma+}, 
$\psi((b,x))=\gamma(gb\sigma(g),\sigma(g)^{\#}xq)$ for $\gamma\in k^*$ and $(g,q)\in H_0$. By Proposition \ref{yanchevskii}, we can write $g=zs$ with $s,z\in B^*$ satisfying 
$\sigma_u(z)z=z\sigma_u(z),~\sigma_u(s)=s$. Then, since $N_B(\sigma(s)^{-1}s)=\overline{N_B(s)}^{-1}N_B(s)=1$, it follows that $(s,1)\in H_0$. 

Since the scalar homotheties in {\bf Str}$(A,S)$ constitute a $1$ dimensional $k$-split torus, by (\cite{Vos}, \S 16, Prop. 2, 16.2, Prop. 3), scalar multiplications over $k$ are in $R(${\bf Str}$(A,S)(k))$. 

Hence we may assume $\psi((b,x))=(gb\sigma(g), \sigma(g)^{\#}xq)$. Now, using $g=zs$ as above (with $\sigma_u(s)=s,~\sigma_u(z)z=z\sigma_u(z)$) and that $(s,1)\in H_0$, writing $\theta$ for the 
norm similarity $(b,x)\mapsto (sb\sigma(s),\sigma(s)^{\#}x)$, we have  $\psi=\psi'\circ\theta$, where $\psi'\in Str(A,S)$ with $\psi'((b,x)):=(zb\sigma(z),\sigma(z)^{\#}xq)$,~$\sigma_u(z)z=z\sigma_u(z)$ and $(z,q)\in H_0$. 
\vskip1mm
\noindent
{\bf Claim :} $\theta\in R(${\bf Str}$(A,S)(k))$. 
To see this, define $s_t:=(1-t)s+t\in B_{\overline{k}}$ and let $\lambda:\mathbb{A}^1_k\rightarrow$ {\bf Str}$(A,S)$ be defined by 
$$\lambda(t)((b,x))=(s_tb\sigma(s_t),\sigma(s_t)^{\#}x),~(b,x)\in A_{\overline{k}}=J(B_{\overline{k}},\sigma, u, \mu)=(B_{\overline{k}},\sigma)_+\oplus B_{\overline{k}}.$$
Then $\lambda$ is a rational map, defined over $k$, regular at $0$ and $1$ and $\lambda(0)=\theta$, 
$\lambda(1)=1$ in {\bf Str}$(A,S)$. This settles the claim.

Hence it only remains to show that $\psi'\in R(${\bf Str}$(A,S)(k))$. For this, observe that $p:=\sigma_u(z)^{-1}z\in U(B,\sigma_u)$ since $z\sigma_u(z)=\sigma_u(z)z$. We have, 
$$N_B(p)=\overline{N_B(z)}^{-1}N_B(z)=\overline{N_B(gs^{-1})}^{-1}N_B(gs^{-1})=N_B(\sigma(g)^{-1}g)=N_B(q).$$
Hence $q=pq'$ for some $q'\in SU(B,\sigma_u)$. We therefore have $(z,q)=(z,pq')\in H_0$ for some $q'\in SU(B,\sigma_u)$.

Hence $(z,q)=(z,p)(1,q')$ and $(1,q')\in H_0$. Also $\text{\bf SU}(B,\sigma_u)\subset $ {\bf Str}$(A,S)$ via $\eta\mapsto[(b,x)\mapsto (b,x\eta)],~\eta\in\text{\bf SU}(B,\sigma_u)$. By (\cite{CP}), {\bf SU}$(B,\sigma_u)$ is rational, hence $R$-trivial. Hence the automorphism $\psi_0':(b,x)\mapsto (b,xq')$ belongs to $R(${\bf Str}$(A,S)(k))$. \\
\noindent
Let $\psi_0\in Str(A,S)$ be defined by $\psi_0(b,x):=(zb\sigma(z),\sigma(z)^{\#}xp)$. Then $\psi'=\psi_0'\circ\psi_0$. 
\vskip1mm
\noindent
{\bf Claim :} $\psi_0\in R(${\bf Str}$(A,S)(k))$. 
To see this, define $z_t=(1-t)z+t\in B_{\overline{k}}$. Then $\sigma_u(z_t)z_t=z_t\sigma_u(z_t)$. Let $q_t=\sigma_u(z_t)^{-1}z_t$ whenever $z_t\in B_{\overline{k}}^*$. Then $q_t\in\text{\bf U}(B,\sigma_u)$ and the map 
$$f:\mathbb{A}^1_k\rightarrow\text{\bf Str}(A,S),~~f(t)((b,x))=(z_tb\sigma(z_t),\sigma(z_t)^{\#}xq_t)$$ 
is defined on an open subset of $\mathbb{A}^1_k$, is a rational map and $f(t)\in$ {\bf Str}$(A,S)$ corresponds to $(z_t,q_t)\in H_0$. In particular, $f(0)$ corresponds to the pair $(z_0,q_0)=(z,p)$, hence $f(0)=\psi_0$. Now $f(1)$ corresponds to 
$(z_1, q_1)=(1,q_1)\in H_0$. But $q_1=\sigma_u(z_1)^{-1}z_1=1$ since $z_1=1$. Hence $f(1)=1\in$ {\bf Str}$(A,S)$. This settles the claim and the proof of the theorem is complete. 
\end{proof}
\begin{corollary}\label{Aut-R-triv-S} Let $A$ be an Albert (division) algebra and $S$ a $9$-dimensional subalgebra of $A$. Then $Aut(A,S)\subset R(\text{\bf Str}(A)(k))$.
\end{corollary}
\noindent
We now can prove
\begin{theorem}\label{Str-R-triv-L} Let $A$ be a first Tits construction Albert division algebra and $L\subset A$ a cubic cyclic subfield. Then $Str(A,L)\subset R(${\bf Str}$(A)(k))$.
\end{theorem}
\begin{proof} By Theorem \ref{str-gate}, $Str(A,L)\subset C.Instr(A).H$. By Theorem \ref{Instr} it follows that $C.Instr(A)\subset R(${\bf Str}$(A)(k))$, where $C$ is the subgroup of $Str(A)$ consisting of scalar homotheties and $H$ is the subgroup of $Aut(A)$ generated by all automorphisms stabilizing $9$-dimensional subalgebras of $A$. So it suffices to prove $H\subset R(${\bf Str}$(A)(k))$. 
Any element of $H$ is a product of automorphisms, each stabilizing some $9$-dimensional subalgebra of $A$. Let $h\in H$ and write $h=h_1h_2\cdots h_r$ with $h_i\in Aut(A,S_i)$, $S_i$ a $9$-dimensional subalgebra of $A$. But by the above theorem $Aut(A,S_i)\subset Str(A,S_i)\subset R(${\bf Str}$(A)(k))$. Hence $h\in R(${\bf Str}$(A)(k))$ and the proof is complete.   
\end{proof}
\begin{corollary}\label{cor-R-triv-L} Assume that $A$ is a first Tits construction and $L\subset A$ a cyclic cubic subfield. Then $Aut(A,L)\subset R(\text{\bf Str}(A)(k))$.
\end{corollary}
\begin{proof} We have $Aut(A,L)\subset Str(A,L)$, hence the result follows from Theorem \ref{Str-R-triv-L}.
\end{proof}
\noindent
\begin{theorem}\label{aut-R-triv} Let $A$ be an Albert division algebra, arising from the first Tits construction. Then $Aut(A)\subset R($\text{\bf Str}$(A)(k))$.
\end{theorem}
\begin{proof} Fix a cubic cyclic extension $E\subset A$, which is possible by (\cite{PR3}). Let $\phi\in Aut(A)$. If $\phi(E)=E$, then $\phi\in Aut(A,E)$ and by Theorem \ref{Str-R-triv-L}, 
$\phi\in R($\text{\bf Str}$(A)(k))$. 

So we may assume 
$F:=\phi(E)\neq E$. Let $S$ be the subalgebra of $A$ generated by $E$ and $F$. Then (cf. \cite{PR6}, Thm. 1.1), $S$ is $9$-dimensional. Hence $S=(B,\sigma)_+$ for a degree $3$ central simple algebra $B$ over 
a quadratic \'{e}tale extension $K$ of $k$, with an involution of the second kind over $K$. Hence we may write $A=J(B,\sigma,u,\mu)$ for suitable parameters. If $Z(B)=K=k\times k$, then we have $(B,\sigma)_+=D_+$ for a degree $3$ central division algebra $D$ over $k$, and $A=J(D,\mu_0)$ for some $\mu_0\in k^*$. 

Then, 
by the classical Skolem-Noether theorem, $\phi|E:E\rightarrow F$ extends to an automorphism of $D_+$ and hence, by Corollary \ref{first-aut-ext}, to an automorphism $\theta$ of $A$, $\theta(D_+)=D_+$ and $\theta|E=\phi|E$. Hence $\psi:=\phi^{-1}\circ\theta$ fixes $E$ pointwise. By Corollary \ref{cor-R-triv-L}, $\psi\in Aut(A,E)\subset R($\text{\bf Str}$(A)(k))$. Also, $\theta\in Aut(A,D_+)$. Hence by Corollary \ref{Aut-R-triv-S}, 
$\theta\in Aut(A,D_+)\subset R($\text{\bf Str}$(A)(k))$. It follows that $\phi\in R($\text{\bf Str}$(A)(k))$. 

So now we assume $K$ is a field. First extend $\phi|E$ to $\phi':EK\rightarrow FK\subset B$ to a $K$-linear isomorphism. By Skolem-Noether theorem, there exists $g\in B^*$ such that $\phi'=Int(g):EK\rightarrow FK$. Since $\phi(E)=F$, it follows that $Int(g)(E)=F=\phi(E)\subset(B,\sigma)_+$. Therefore 
$$\sigma(gxg^{-1})=gxg^{-1}~\forall x\in E.$$ 
This yields 
$$\sigma(g)^{-1}x\sigma(g)=gxg^{-1}~\forall x\in E.$$
Hence $\sigma(g)g$ centralizes $E$ and hence $EK$ in $B$. Since $EK$ is maximal commutative in $B$, it follows that $\sigma(g)g\in EK$. Also $\sigma(g)g\in (B,\sigma)_+$. Hence $\sigma(g)g\in E$. 
Similarly, it follows that $g\sigma(g)\in F$. 

Let $e:=\sigma(g)g\in E$ and $f:=g\sigma(g)\in F$. Define $\psi\in Str(S)$ by 
$$\psi(x)=gx\sigma(g),~x\in S.$$
Then for any $x\in E$ we have $xe\in E$. Hence, 
$$\psi(x)=gx\sigma(g)=gxeg^{-1}=\phi(xe)\in\phi(E)= F.$$
Therefore $\psi(E)=F$ and $\psi(1)=f\in F$. Since $A$ is a first Tits construction and $F$ is cyclic, there is a degree $3$ central division algebra $D$ over $k$, such that $F\subset D_+\subset A$ (see \cite{PR2}, Cor. 4.5). 

By (\cite{Th-1}, Lemma 5.1), there exists $\chi\in C.Instr(A)$ such that $\chi(f)=1,~\chi(D_+)=D_+$ and $\chi(F)=F$. Hence $\chi\circ\psi(1)=1$. 

We extend $\psi$ to an element $\widetilde{\psi}\in Str(A,S)$ by the formula in Theorem \ref{ext-norm-sim}. Then $\chi\circ\widetilde{\psi}\in Str(A)$ and $\chi\circ\widetilde{\psi}(1)=1$. Hence $\chi\circ\widetilde{\psi}\in Aut(A)$ with $\chi\circ\widetilde{\psi}(E)=F$. Therefore $\phi^{-1}\circ\chi\circ\widetilde{\psi}\in Aut(A)$ and $\phi^{-1}\circ\chi\circ\widetilde{\psi}(E)=E$. Hence, by Theorem \ref{Str-R-triv-L}, 
$\phi^{-1}\circ\chi\circ\widetilde{\psi}\in R($\text{\bf Str}$(A)(k))$. Also, $\chi\in Str(A,D_+)$, and $\widetilde{\psi}\in Str(A,S)$. We have already shown in Theorem \ref{Str-R-triv} that these are subgroups of $R($\text{\bf Str}$(A)(k))$. It now follows that $\phi\in R($\text{\bf Str}$(A)(k))$. 
\end{proof} 
\begin{lemma} Let $L$ be a cubic extension of $k$ and $\lambda\in k*$. Let $S$ denote the first Tits process $J(L,\lambda)$ over $k$. Let $a\in L*$. Define $\chi\in Str(S)$ by 
$$ \chi=R_{N_L(a)}U_{(0,0,1)}U_{(0, N_L(a)^{-1}, 0)}.$$
Then $\chi(a)=1$.
\end{lemma}
\begin{proof} The proof follows from a calculation exactly same as in the proof of (Lemma 5.1, \cite{Th-1}).
\end{proof} 
We are now in a position to prove our main result:
\begin{theorem}\label{Main} Let $A$ be an Albert division algebra arising from the first Tits construction. Then {\bf Str}$(A)$ is $R$-trivial.
\end{theorem}
\begin{proof} Let $\psi\in Str(A)$ and $a:=\psi(1)$. If $a\in k^*$, we have $R_a\in R($\text{\bf Str}$(A)(k))$ and $R_{a^{-1}}\circ\psi(1)=1$. Hence, by Theorem \ref{aut-R-triv}, $R_{a^{-1}}\circ\psi\in Aut(A)\subset R($\text{\bf Str}$(A)(k))$. This implies $\psi\in R($\text{\bf Str}$(A)(k))$. So we may assume $a\notin k^*$. Then $L:=k(a)\subset A$ is a cubic separable subfield and, by (\cite{PR2}, Cor. 4.5), there is $\lambda\in k^*$ such that the Tits process $S=J(L,\lambda)\subset A$. Define $\chi\in Str(S)$ by 
$$\chi=R_{N(a)}U_{(0,0,1)}U_{(0,N(a)^{-1}a,0)}.$$
By the above lemma $\chi(a)=1$. We now extend $\chi$ to $\widetilde{\chi}\in Str(A)$ using $(*)$. Then $\widetilde{\chi}(S)=S$ and hence by Theorem \ref{Str-R-triv}, $\widetilde{\chi}\in R($\text{\bf Str}$(A)(k))$. Also,   
$$\widetilde{\chi}\circ\psi(1)=\widetilde{\chi}(a)=1.$$
Hence $\widetilde{\chi}\circ\psi\in Aut(A)\subset R($\text{\bf Str}$(A)(k)) $ and it follows that $\psi\in R($\text{\bf Str}$(A)(k))$.  
\end{proof} 
\begin{theorem}\label{Isom} Let $A$ be an Albert division algebra arising from the first Tits construction over a field $k$. Then {\bf Isom}$(A)$ is $R$-trivial. 
\end{theorem}
\begin{proof} Let $\psi\in Isom(A)=$ {\bf Isom}$(A)(k)\subset Str(A)=$ {\bf Isom}$(A)(k)$. By Theorem \ref{Main}, {\bf Str}$(A)$ is $R$-trivial. 
Let $\theta:\mathbb{A}^1_k\rightarrow$ {\bf Str}$(A)$ be a rational map connecting $\psi$ to $1$ in {\bf Str}$(A)$, $\theta(0)=\psi,~\theta(1)=1$. Let $\psi_t:=\theta(t),~a_t:=\psi_t(1),~\lambda_t:=N(a_t)$. Let $\chi_t:=R_{\lambda_t^{-1}} U_{a_0}^{-1}U_{a_t}\psi_t$. Then 
$$\chi_t(1)=\lambda_t^{-1}U_{a_0}^{-1}U_{a_t}\psi_t(1)=\lambda_t^{-1}U_{a_0}(a_t^3).$$
Since 
$$N(a_0)=N(\psi_0(1))=N(\psi(1))=1,$$
we have
$$N(\chi_t(1))=\lambda_t^{-3}N(a_0)^2N(a_t)^3=1,$$
It follows that $\chi_t\in$ {\bf Isom}$(A)$. Also, $t\mapsto \chi_t$ is a rational map $\mathbb{A}^1_k\rightarrow$ {\bf Isom}$(A)$. Now $\lambda_0=N(\psi_0(1))=N(\psi(1))=1$ as $\psi$ is a norm isometry. Hence 
$$\chi_0=U_{a_0}^{-1}U_{\psi(1)}\psi_0=\psi,$$ 
and $\chi_1=1$. This proves the assertion.
\end{proof}
\section{\bf Tits-Weiss conjecture and $E^{78}_{8,2}$ }
In this section, we will prove the Tits-Weiss conjecture for Albert division algebras that arise from the first Tits construction. Recall that in this paper, we denote by $Str(A)$ the group {\bf Str}$(A)(k)$ of $k$-rational points of the algebraic group {\bf Str}$(A)$, the full group of {\bf norm similarities} of $A$. We denote by {\bf Isom}$(A)$ the full group of {\bf norm isometries} of $A$ and $Isom(A):=\text{\bf Isom}(A)(k)$. By (\cite{SV}, Thm. 7.3.2) {\bf Isom}$(A)$ is a connected, simple, simply connected algebraic group of type $E_6$ defined over $k$. It follows that {\bf Isom}$(A)$ is the commutator subgroup of {\bf Str}$(A)$. Let $C$ denote the subgroup of $Str(A)$ consisting of all scalar homotheties $R_a,~a\in k^*$. \\
\noindent
{\bf Tits-Weiss conjecture :} Let $A$ be as above, then $Str(A)=C.Instr(A)$. 
\vskip0.5mm
\noindent
Equivalently, the conjecture predicts that $\frac{Str(A)}{C.Instr(A)}=\{1\}$. This is equivalent to the {\bf Kneser-Tits conjecture} for groups of type $E^{78}_{8,2}$. 

Let {\bf G} be a simple, simply connected group defined over $k$, with Tits index $E^{78}_{8,2}$ and {\bf G}$(k)^{\dagger}$ be the subgroup of {\bf G}$(k)$ generated by the $k$-rational points of the unipotent radicals of parabolic $k$-subgroups of {\bf G}. The Kneser-Tits conjecture for {\bf G} predicts that $\frac{\text{\bf G}(k)}{\text{\bf G}(k)^{\dagger}}=\{1\}$. 

The (reductive) anisotropic kernel of {\bf G} is the structure group {\bf Str}$(A)$ of a uniquely (upto isotopy) determined Albert division algebra $A$ defined over $k$ (see \cite{T1}, 3.3.1). It is proved in (\cite{TW}, 37.41, 42.3.6) that $\frac{\text{\bf G}(k)}{\text{\bf G}(k)^{\dagger}}\cong\frac{Str(A)}{C.Instr(A)}$. Hence, to prove the Tits-Weiss conjecture, we need to show the first quotient is trivial. By (Thm. 7.2, \cite{G}), it suffices to prove {\bf G} is $R$-trivial. In fact, it suffices to prove $W(k,G)=G(k)/G(k)^{\dagger}\cong G(k)/R=\{1\}$.

Let {\bf S} be a maximal $k$-split torus in {\bf G}. By the Bruhat-Tits decomposition for {\bf G}, it follows that there is a birational $k$-isomorphism $\text{\bf G}\cong\text{\bf U}\times Z_{\text{\bf G}}(\text{\bf S})\times\text{\bf U}$, where {\bf U} is the unipotent radical of a minimal parabolic $k$-subgroup containing {\bf S}. It is well known that the underlying variety of {\bf U} is rational over $k$. Hence, by (\cite{Vos}, Prop. 1, \S 16), it follows that 
$$\frac{\text{\bf G}(k)}{R}\cong \frac{Z_{\text{\bf G}}(\text{\bf S})(k)}{R}.$$ 
Again, by (\cite{G2}, 1.2), 
$$ \frac{Z_{\text{\bf G}}(\text{\bf S})(k)}{R}\cong\frac{Z_{\text{\bf G}}(\text{\bf S})}{\text{\bf S}}(k)/R.$$
Hence, to prove the Tits-Weiss conjecture, it suffices to prove that the second quotient above is trivial. We now proceed to prove this. Note that $\text{\bf H}:=Z_{\text{\bf G}}(\text{\bf S})$ is a connected reductive $k$-subgroup of {\bf G}, is the reductive anisotropic kernel of {\bf G} and its commutator subgroup $\text{\bf H}'$ is simple $k$-anisotropic, strongly inner of type $E_6$. Hence, by (\cite{T1}, 3.3), $\text{\bf H}'=\text{\bf Isom}(A)$. We first prove
\begin{lemma} The connected center of {\bf H} equals $S$.
\end{lemma}
\begin{proof} We know that {\bf H} is connected reductive and its connected center contains {\bf S}. We have $\text{\bf H}=\text{\bf H}'.Z(\text{\bf H})^{\circ}$, an almost direct product of groups. Since $\text{\bf H}'$ is simple of type $E_6$, $Rank(\text{\bf H}')=6$. Let {\bf T} be any maximal $k$-torus in $\text{\bf H}'$. The connected center $Z(\text{\bf H})^{\circ}$ of {\bf H} is a torus 
$\text{\bf S}'\supset\text{\bf S}$ and $\text{\bf S}'\cap\text{\bf T}$ is finite. Hence $\text{\bf T}.\text{\bf S}'$ is a torus in {\bf H} of dimension $6+rank(\text{\bf S}')$ and 
$rank(\text{\bf S}')\geq 2$. Therefore
$$Rank(\text{\bf G})=8\geq 6+ rank(\text{\bf S}').$$
It follows that $\text{\bf S}'=\text{\bf S}$.  
\end{proof}
\noindent
In light of this lemma, we have 
$$\frac{Z_{\text{\bf G}}(\text{\bf S})}{\text{\bf S}}(k)/R=\frac{\text{\bf H}'\text{\bf S}}{\text{\bf S}}(k)/R\cong\frac{\text{\bf H}'}{\text{\bf H}'\cap\text{\bf S}}(k)/R.$$
\noindent
It is well known (see Thm. 7.3.2, \cite{SV}. Thm. 14.27, \cite{Spr-J}) that for an Albert algebra $A$ over $k$, {\bf Str}$(A)$ is a connected reductive group defined over $k$. We need information on its center:  
\begin{lemma} Let $A$ be an Albert algebra over $k$. Then the group $\text{\bf Str}(A)$ is connected reductive with center $\mathbb{G}_m$. 
The commutator $[\text{\bf Str}(A),\text{\bf Str}(A)]=\text{\bf Isom}(A)$ and the sequence 
$$\{1\}\rightarrow\text{\bf Isom}(A)\rightarrow\text{\bf Str}(A)\rightarrow\mathbb{G}_m\rightarrow\{1\}$$
is exact, where $\text{\bf Str}(A)\rightarrow\mathbb{G}_m$ maps a norm similarity to its factor of similitude. 
\end{lemma}
\begin{proof} That $\text{\bf Str}(A)$ is connected reductive is proved in (3.3.1, \cite{T1}, \cite{SV}). Let $\psi\in Z(\text{\bf Str}(A))$. We choose two $9$-dimensional degree $3$ separable $k$-subalgebras of $A$, say $S_1$ and $S_2$ such that $S_1\cap S_2=k$ (see \cite{P-S-T2}, Lemma 5.2). Let $\theta_i\in \text{\bf Aut}(A/S_i)$ be such that $A^{\theta_i}=S_i$. Then since $\psi$ commutes with $\theta_i,~i=1,2$, $\psi$ must stabilize both $S_1$ as well as $S_2$. Hence $\psi$ stabilizes $S_1\cap S_2=k$. Thus $a:=\psi(1)\in k^*$ and $\phi=R_{a^{-1}}\circ\psi\in\text{\bf Aut}(A)$. Hence $\phi\in Z(\text{\bf Aut}(A))$. 
But $\text{\bf Aut}(A)$ is a group of type $F_4$, hence has trivial center. It follows that $\psi=R_a$. Hence $Z(\text{\bf Str}(A))=\mathbb{G}_m$, the $k$-subgroup of $\text{\bf Str}(A)$ consisting of all scalar homotheties $R_a,~a\in\mathbb{G}_m$. The remaining assertions follow from (\cite{SV}, Thm. 7.3.2) and (\cite{Spr-J}. Thm. 14.27).   
\end{proof}
\noindent
Now we observe that maximal $k$-split tori in {\bf G} are conjugate by $\text{\bf G}(k)$ and $\text{\bf G}\supset\text{\bf Str}(A)$ over $k$ (see \cite{TW}, 42.6). Hence we may assume that 
$\text{\bf S}\supset Z(\text{\bf Str}(A))$. Denote by {\bf Z} the center of {\bf Str}$(A)$. We have 
$$\text{\bf Str}(A)=[\text{\bf Str}(A),\text{\bf Str}(A)].\text{\bf Z}=\text{\bf Isom}(A).\text{\bf Z}.$$
Since both $\text{\bf S}$ and {\bf Z} are split over $k$, by (\cite{B}, 8.5. Corollary) we have a decomposition 
$\text{\bf S}=\text{\bf Z}.\text{\bf S}'$ of $k$-tori, with 
$\text{\bf Z}\cap\text{\bf S}'=\{1\}$. 
Hence, since $\text{\bf H}'=[\text{\bf Str}(A),\text{\bf Str}(A)]=\text{\bf Isom}(A)$, we have  
$$\text{\bf H}'\text{\bf S}=\text{\bf H}'\text{\bf Z}.\text{\bf S}'=\text{\bf Isom}(A)\text{\bf Z}.\text{\bf S}'=\text{\bf Str}(A).\text{\bf S}'.$$
Also, since $\text{\bf Z}\cap\text{\bf S}'=\{1\}$ and $\text{\bf S}'\subset\text{\bf S}=Z(\text{\bf H})^{\circ}$, it follows that $\text{\bf Str}(A)\cap\text{\bf S}'=\{1\}$. 
We therefore have
$$\frac{Z_{\text{\bf G}}(\text{\bf S})}{\text{\bf S}}=\frac{\text{\bf H}'\text{\bf S}}{\text{\bf S}}\cong\frac{\text{\bf Str}(A)\times\text{\bf S}'}{\text{\bf Z}\times\text{\bf S}'}
\cong\frac{\text{\bf Str}(A)}{\text{\bf Z}}.$$
\noindent
Finally, by Hilbert Theorem-$90$, $H^1(k,\text{\bf Z})=\{1\}$. Hence, it follows that 
$$\frac{\text{\bf Str}(A)}{\text{\bf Z}}(k)\cong\frac{\text{\bf Str}(A)(k)}{\text{\bf Z}(k)}.$$
\noindent
We can now state and prove
\begin{theorem}\label{Weiss} Let {\bf G} be a simple, simply connected algebraic group over $k$ with Tits index $E^{78}_{8,2}$, whose anisotropic kernel is split by a cubic extension of $k$. Then {\bf G} is $R$-trivial.   
\end{theorem}
\begin{proof} Let {\bf H} denote the reductive anisotropic kernel of $G$ as in the above discussion. Then, by Corollary \ref{an-kernel}, $[\text{\bf H},\text{\bf H}]={\bf Isom}(A)$ for an Albert division algebra $A$ over $k$, arising from a first Tits construction. We have already shown 
$$\text{\bf G}(k)/R\cong\frac{\text{\bf Str}(A)}{\text{\bf Z}}(k)/R\cong\frac{\text{\bf Str}(A)(k)}{\text{\bf Z}(k)}/R.$$
\noindent
Also, by Theorem \ref{Main}, when $A$ is a first Tits construction, $\text{\bf Str}(A)$ is $R$-trivial. Let $L$ be any field extension of $k$. If $A\otimes_k L$ is split, then $\text{\bf G}$ splits over $L$ and hence $W(L,\text{\bf G})=\{1\}$. If $A\otimes_k L$ is division, then the Tits index of $\text{\bf G}$ over $L$ does not change and the anisotropic kernel of $\text{\bf G}$ over $L$ corresponds to {\bf Str}$(A\otimes_kL)$. Hence 
$$W(L,\text{\bf G})= \text{\bf G}(L)/R=\frac{\text{\bf Str}(A)(L)}{\text{\bf Z}(L)}/R=\{1\},$$
the last equality holds because $\text{\bf Str}(A)$ is $R$-trivial : since $\frac{\text{\bf Str}(A)}{\text{\bf Z}}(L)=\frac{\text{\bf Str}(A)(L)}{\text{\bf Z}(L)}$, any element of $\frac{\text{\bf Str}(A)}{\text{\bf Z}}(L)$ is represented by an element of $\text{\bf Str}(A)(L)$ and any element of $\text{\bf Str}(A)(L)$ can be joined to the identity element by an $L$-rational image of $\mathbb{A}^1_L$, as $\text{\bf Str}(A)$ is $R$-trivial. The result now follows.
\end{proof}
\vskip1mm
\noindent
{\bf Remark :} Note that we have proved above also that $\frac{\text{\bf Str}(A)}{\text{\bf Z}}$, the {\bf adjoint} form of $E_6$ corresponding to a first Tits construction Albert algebra $A$ over a field $k$ is $R$-trivial. 
\begin{corollary}\label{TW}{\bf Tits-Weiss Conjecture :} Let $A$ be a first Tits construction Albert division algebra. Then $Str(A)=C.Instr(A)$.
\end{corollary}
\begin{proof}This follows immediately from the above theorem and that $\frac{Str(A)}{C.Instr(A)}\cong\frac{G(k)}{R}$, where $G$ is as in the above theorem.
\end{proof}
\section{\bf Groups with index $E^{78}_{7,1}$}
In this brief section, we wish to explore $R$-triviality issues for simple groups of type $E^{78}_{7,1}$. The reason for considering such groups is that these groups also have their anisotropic kernels 
as the structure groups of Albert division algebras and all groups of this type arise this way (see \cite{T1}, 3.3.1). 

We will consider the groups ${\bf \Xi}$ with this index whose anisotropic kernel is the structure group of an Albert division algebra, that is a Tits first construction. These groups were studied by Max Koecher in (\cite{Ko}). We use the explicit description of ${\bf \Xi}$ given by Max Koecher. 

Fix an Albert division algebra $A$ over $k$. Let ${\bf \Xi}(A)$ be the full group of birational transformations generated by $j:A_{\overline{k}}\rightarrow A_{\overline{k}},~j(x):=(-x)^{-1}$ and the translations $t_a:~x\mapsto a+x,~a,x\in A_{\overline{k}}$. Then every element of ${\bf \Xi}(A)$ has an expression $f=w\circ t_a\circ j\circ t_b\circ j\circ t_c,~a,b,c\in A_{\overline{k}}$. 

In particular $\text{\bf Str}(A)$ is contained in 
${\bf \Xi}(A)$ and $\text{\bf Isom}(A)$ is the anisotropic kernel (up to a central torus) of ${\bf \Xi}(A)$. More precisely, let {\bf Z} denote, as before, the center of $\text{\bf Str}(A)$, which is a $\mathbb{G}_m$ defined over $k$. Let $\text{\bf S}\subset{\bf \Xi}(A)$ be a maximal $k$-split torus containing {\bf Z}. Then $\text{\bf H}=Z_{\bf \Xi}(\text{\bf S})$ is a connected reductive $k$-subgroup of $G$, is the reductive anisotropic kernel of ${\bf\Xi}={\bf \Xi}(A)$. Since $k$-rank of ${\bf \Xi}$ is $1$, we have $\text{\bf S}=\text{\bf Z}$. It follows from this that $Z(\text{\bf H})^{\circ}=\text{\bf Z}$. We have, as in the previous section
$${\bf\Xi}(k)/R\cong Z_{\bf\Xi}(\text{\bf S})(k)/R\cong\frac{Z_{\bf\Xi}(\text{\bf S})}{\text{\bf S}}(k)/R.$$
\noindent
Now, 
$$\text{\bf H}=Z_{\bf\Xi}(\text{\bf S})=\text{\bf H}'.Z(\text{\bf H})^{\circ}=\text{\bf Isom}(A).\text{\bf Z}=\text{\bf Str}(A).$$
\noindent
Hence 
$$\frac{Z_{\bf\Xi}(\text{\bf S})}{\text{\bf S}}=\frac{\text{\bf Str}(A)}{\text{\bf Z}}.$$
We have shown in the last section that $\frac{\text{\bf Str}(A)}{\text{\bf Z}}(k)/R=\{1\}$. Hence we have, by arguments quite analougous to the proof of $R$-triviality of groups with index $E^{78}_{8,2}$, 
\begin{theorem}\label{Koecher} Let ${\bf\Xi}$ be a simple, simply connected algebraic group over $k$ with Tits index $E^{78}_{7,1}$, whose semisimple anistropic kernel is split over a cubic extension of $k$. Then ${\bf\Xi}$ is $R$-trivial. 
\end{theorem} 
\noindent{\bf Reduced Albert algebras :} In this mini-section, we wish to explore the group {\bf Str}$(A)$ when $A$ is a reduced Albert algebra, mainly the $R$-triviality of such groups. We recall at this stage our result from (\cite{Th-1}), which. in fact, holds over fields of arbitrary characteristics :
\begin{theorem}\label{reduced-tits-weiss} Let $A$ be a reduced Albert algebra over any field $k$. Then $Str(A)=C.Instr(A)$, where $C$ denotes the subgroup of $Str(A)$ of scalar homotheties of $A$ and $Instr(A)$ is the inner structure group. 
\end{theorem}
\begin{proof} This follows from a result of J. R. Faulkner (\cite{F1}, Thm. 1, \cite{F2}), that for any reduced Albert algebra $A$, the group $Isom(A)$ is simple modulo its center. Given $\psi\in Str(A)$, let $a:=\psi(1)$ and $\alpha:=N(a)$. Then $\alpha\neq 0$ and $U_a$ is invertible. We now consider $\chi:=R_{\alpha^{-1}}U_a\psi\in Str(A)$. Then $\chi\in Isom(A)$. Since $Instr(A)$ is a normal subgroup of $Str(A)$ and norm similarities of the form $U_{a_1}U_{a_2}...U_{a_r}\in Isom(A)$ whenever $\Pi N(a_i)=1$, it follows that $(Instr(A)\cap Isom(A)).Z\neq \{1\}$ and hence 
$$(Instr(A)\cap Isom(A)).Z=Isom(A),$$ 
where $Z$ is the center of $Isom(A)$ and consists of scalar homotheties $R_c$ with $c\in k^*,~c^3=1$. Therefore $\chi\in C. Instr(A)$ and hence $\psi\in C.Instr(A)$. This completes the proof.
\end{proof}
We have already shown in \S 4 that $Instr(A)\subset R(${\bf Str}$(A)(k))$. Hence it follows from the above theorem that $Str(A)\subset R(${\bf Str}$(A)(k))$. This inclusion is independent of the base field, hence 
$\text{\bf Str}(A)$ is $R$-trivial. We record this as 
\begin{theorem}\label{main-reduced} Let $A$ be a reduced Albert algebra over a field $k$. Then {\bf Str}$(A)$ is $R$-trivial.
\end{theorem}
\begin{theorem} Let $A$ be an Albert algebra arising from the first Tits construction. Then any norm similarity of $A$ is a product of norm similarities, each stabilizing a $9$-dimensional subalgebra of $A$.
\begin{proof} First assume $A$ is split. Let $\psi\in Str(A)$ and $\psi(1)=a$. Then, $\chi:=R_{N(a)^{-1}}\circ U_a\circ\psi\in Isom(A)$ and $R_{N(a)^{-1}}\circ U_a\in Str(A)$ stabilizes any subalgebra of $A$ containing $a$. 
By a result of Faulkner cited above, $Isom(A)$ is a simple group modulo its center, which contains scalar homotheties. Hence $Isom(A)$ is generated by norm isometries which stabilize $9$-dimensional subalgebras and the result follows. Now assume $A$ is a division algebra. As above, we may assume $\psi\in Isom(A)$. As in the proof of Theorem \ref{Main}, we reduce to the case when $\psi\in Aut(A)$. Again, using the arguments in the proof of Theorem \ref{aut-R-triv}, we may assume that $\psi$ stabilizes a cubic cyclic extension $E\subset A$, i.e. $\psi\in Str(A,E)$. By Theorem \ref{str-gate}, 
$Str(A,E)\subset C. Instr(A).H$, where $C$ is the subgroup of scalar homoteties on $A$, $Instr(A)$ is the subgroup of $Str(A)$ generated by the $U$-operators of $A$ and $H$ is the subgroup of $Aut(A)$ generated by automorphisms that stabilize $9$-dimensional subalgebras of $A$. Hence $\psi$ admits the required factorization.  
\end{proof}
\end{theorem}
\begin{theorem}\label{Isom-reduced} Let $A$ be a reduced Albert algebra over a field $k$. Then {\bf Isom}$(A)$ is $R$-trivial. 
\end{theorem}
\begin{proof} Let $\psi\in Isom(A)=$ {\bf Isom}$(A)(k)\subset Str(A)$. By Theorem \ref{main-reduced}  {\bf Str}$(A)$ is $R$-trivial. Now, by an argument exactly as in the proof of Theorem \ref{Isom}, the result follows. 
\end{proof}
\section{\bf Retract rationality} Recall that a an irreducible variety $X$ defined over a field $k$ is {\bf retract $k$-rational} if there exists a non-empty open subset $U\subset X$, defined over $k$ such that the identity map of $U$ factors through an open subset $V$ of an affine space $\mathbb{A}^m_k$, i.e. there exists morphisms $f:U\rightarrow V$ and $r:V\rightarrow U$ such that $r\circ f=id_U$ (see 2.2, \cite{G}). The following result is proved in (\cite{G}, Prop. 5.4, Thm. 5.9):
\begin{theorem}\label{retract}Let $G$ be a semisimple, simply connected, absolutely almost simple group, defined and isotropic over an infinite field $k$, then the following are equivalent
\begin{enumerate}
\item $G$ is $W$-trivial, i.e. $W(F,G)=\{1\}$ for all field extensions $F/k$.
\item $G$ is a retract $k$-rational variety.
\end{enumerate}
\end{theorem}
\noindent
Combining this with our results, we have, 
\begin{corollary} Let $G$ be a simple, simply connected algebraic group defined over a field $k$ with Tits index $E^{78}_{8,2}$ or $E^{78}_{7,1}$ whose anisotropic kernel splits over a cubic extension of $k$. Then $G$ is retract $k$-rational.
\end{corollary}
\begin{proof} By Corollary \ref{an-kernel}, the anisotropic kernel of $G$ is isomorphic to {\bf Isom}$(A)$ for a first Tits construction Albert division algebra $A$ over $k$. Since if $A$ is a first Tits construction, for any field extension $F$ of $k$, $A\otimes_kF$ is either split or a division algebra (see \cite{PR2}, Cor. 4.2). Let $A_F:=A\otimes_kF$. 
By Theorem \ref{reduced-tits-weiss}, $Str(A_F)=C.Instr(A_F)$ when $A_F$ is split and by our result on Tits-Weiss conjecture, $Str(A_F)=C.Instr(A_F)$ also when $A_F$ is a division algebra. Hence, by Theorem \ref{Weiss} and Theorem \ref{Koecher} we have $W(F,G)=\{1\}$. This proves the assertion on $G$.
\end{proof}
\noindent
\begin{theorem} Let $A$ be a reduced Albert algebra over a field $k$. Then {\bf Isom}$(A)$ is retract rational.
\end{theorem}
\begin{proof} Note that when $A$ is a reduced Albert algebra, the algebraic group {\bf Isom}$(A)$ is simple, simply connected and $k$-isotropic. The assertion now follows from the $R$-triviality of this group, settled in Theorem \ref{Isom-reduced}.
\end{proof}
\vskip1mm
\noindent
{\bf Acknowledgement}\\
\noindent
Part of this work was done when I visited Linus Kramer of the Mathematics Institute, University of Muenster, in March 2016. The stay was supported by the Deutsche Forschungsgemeinschaft 
through SFB 878. His help is gratefully acknowledged. I thank Richard Weiss and Tom De Medts for some very useful discussions and suggestions. I have immensely benefitted from discussions with Bernhard Muhlherr, Gopal Prasad and Dipendra Prasad. I thank Holger Petersson and Otmar Loos for their encouragement and interest in this work.  

\vskip5mm

\end{document}